\documentclass[11pt]{article}

\usepackage{amsmath,amsthm,amssymb,fancyhdr,graphicx,bbm,cancel,color,mathrsfs,todonotes,xypic, pinlabel,graphbox}
\usepackage[all]{xy}

\usepackage{hyperref}

\usepackage{bm}

\usepackage{lipsum}

\usepackage{tabto}

\let\OLDthebibliography\thebibliography
\renewcommand\thebibliography[1]{
  \OLDthebibliography{#1}
  \setlength{\parskip}{0pt}
  \setlength{\itemsep}{0pt plus 0.3ex}
}

\newtheorem{thm}{Theorem}[section]
\newtheorem{mainthm}{Theorem}

\newtheorem{lem}[thm]{Lemma}
\newtheorem{prop}[thm]{Proposition}
\newtheorem{cor}[thm]{Corollary}
\theoremstyle{definition}
\newtheorem{defn}[thm]{Definition}

\newtheorem{prob}[thm]{Problem} 
\newtheorem{rmk}[thm]{Remark}
\theoremstyle{remark}

\newcommand{\dmo}{\DeclareMathOperator}

\newcommand{\R}{\mathbb{R}}
\newcommand{\Q}{\mathbb{Q}}
\newcommand{\co}{\mathbb{C}}\newcommand{\Z}{\mathbb{Z}}

\newcommand{\al}{\alpha}\newcommand{\ga}{\gamma}\newcommand{\ep}{\epsilon}\newcommand{\ze}{\zeta}
\newcommand{\Om}{\Omega}\newcommand{\la}{\lambda}
\newcommand{\Ga}{\Gamma}
\newcommand{\what}{\widehat}\newcommand{\wtil}{\widetilde}\newcommand{\Lam}{\Lambda}
\newcommand{\cd}{\cdots}
\newcommand{\sbs}{\subset}\newcommand{\bs}{\backslash}\newcommand{\pa}{\partial}

\newcommand{\xra}{\xrightarrow}

\newcommand{\ra}{\rightarrow}
\newcommand{\hra}{\hookrightarrow}\newcommand{\onto}{\twoheadrightarrow}
\newcommand{\bb}[1]{\mathbb{#1}}\newcommand{\ca}[1]{\mathcal{#1}}\newcommand{\mf}{\mathfrak}

\newcommand{\fr}[2]{\frac{#1}{#2}}

\newcommand{\op}{\oplus}
\newcommand{\ti}{\times}

\newcommand{\pair}[1]{\langle #1 \rangle}
\dmo{\sgn}{sign}\dmo{\Span}{span}
\dmo{\we}{\wedge}
\dmo{\ind}{ind}\dmo{\Ind}{Ind}
\dmo{\bop}{\bigoplus}\dmo{\pic}{Pic}
\dmo{\coker}{coker}\dmo{\vol}{Vol}\dmo{\gal}{Gal}\dmo{\perm}{Perm}
\dmo{\tor}{Tor}\dmo{\ext}{Ext}\dmo{\Ext}{Ext}
\dmo{\aut}{aut}
\dmo{\Aut}{Aut}
\dmo{\inn}{Inn}\dmo{\var}{Var}
\dmo{\dep}{depth}
\dmo{\ad}{ad}\dmo{\curl}{curl}
\dmo{\hy}{\bb H}
\dmo{\Sl}{SL}
\dmo{\SO}{SO}\dmo{\psl}{PSL}
\dmo{\isom}{Isom}\dmo{\Isom}{Isom}
\dmo{\conf}{Conf}
\dmo{\stab}{Stab}\dmo{\Jac}{Jac }
\dmo{\diam}{diam}\dmo{\fix}{Fixed}\dmo{\Fix}{Fix}
\dmo{\injR}{injRad}\dmo{\Ad}{Ad}
\dmo{\esv}{ess-vol}\dmo{\out}{Out}\dmo{\Out}{Out}
\dmo{\nil}{Nil}\dmo{\sol}{Sol}
\dmo{\Div}{div}
\dmo{\SU}{SU}
\dmo{\SP}{SP}
\dmo{\Sp}{Sp}
\dmo{\rk}{rk}
\dmo{\rank}{rank}
\dmo{\psp}{PSp}\dmo{\psu}{PSU}
\dmo{\PU}{PU}\dmo{\pgl}{PGL}
\dmo{\Mod}{Mod}\dmo{\range}{Range}
\dmo{\eu}{eu}\dmo{\mi}{mi}
\dmo{\Log}{Log}\dmo{\supp}{supp}
\dmo{\maps}{Maps}\dmo{\Gr}{Gr}
\dmo{\Pin}{Pin}
\dmo{\Spin}{Spin}\dmo{\Str}{Str}
\dmo{\Sq}{Sq}\dmo{\Symp}{Symp}
\dmo{\pd}{PD}\dmo{\PD}{PD}\dmo{\sig}{Sig}
\dmo{\Set}{Set}\dmo{\Top}{Top}
\dmo{\ev}{ev}\dmo{\St}{St}
\dmo{\Pt}{Pt}\dmo{\pt}{pt}

\dmo{\colim}{colim }\dmo{\Pl}{PL}
\dmo{\String}{String}\dmo{\smear}{smear}

\dmo{\dev}{dev}
\dmo{\met}{Met}\dmo{\contact}{Contact}
\dmo{\teich}{Teich}\dmo{\Teich}{Teich}\dmo{\qi}{QI}
\dmo{\der}{Der}
\dmo{\cl}{Cliff}\dmo{\Cl}{Cl}
\dmo{\Pf}{Pf}
\dmo{\GL}{GL}
\dmo{\PSL}{PSL}
\dmo{\ch}{ch}\dmo{\diag}{diag}
\dmo{\grad}{grad}\dmo{\Char}{char}
\dmo{\spec}{Spec}\dmo{\Arg}{Arg}
\dmo{\rad}{rad}\dmo{\im}{Im}
\dmo{\Hom}{Hom}\dmo{\End}{End}
\dmo{\tr}{tr}\dmo{\id}{Id}
\dmo{\gl}{GL}
\dmo{\sym}{Sym}\dmo{\Sym}{Sym}
\dmo{\com}{Comm}
\dmo{\Lk}{Lk}
\dmo{\CAT}{CAT}
\dmo{\Rep}{Rep}
\dmo{\Res}{Res}
\dmo{\Conf}{Conf}
\dmo{\PConf}{PConf}
\dmo{\Push}{Push}
\dmo{\Cont}{Cont}
\dmo{\sm}{\setminus}
\dmo{\vn}{\varnothing}
\dmo{\disk}{\mathbb D}
\dmo{\Trd}{Trd}\dmo{\Mat}{Mat}
\dmo{\Riem}{Riem}
\dmo{\Diffn}{\Diff_0}\dmo{\diff}{diff}
\dmo{\Diff}{Diff}\dmo{\homeo}{Homeo}
\dmo{\Homeo}{Homeo}\dmo{\Fr}{Fr}
\dmo{\rot}{rot}\dmo{\Emb}{Emb}
\dmo{\Ham}{Ham}\dmo{\Met}{Met}
\dmo{\Ein}{Ein}\dmo{\CP}{\co P}
\dmo{\Per}{Per}\dmo{\Ric}{Ric}
\newcommand{\C}{\mathbb C}\dmo{\Nrd}{Nrd}
\dmo{\Comp}{Comp}\dmo{\PSC}{PSC}
\dmo{\Cent}{Cent}\dmo{\Orb}{Orb}
\dmo{\aind}{a-ind}\dmo{\tind}{t-ind}
\dmo{\constant}{constant}
\dmo{\Td}{Td}
\dmo{\LMod}{LMod}
\dmo{\SMod}{SMod}
\dmo{\SDiff}{SDiff}
\dmo{\Br}{Br}
\dmo{\csch}{csch}
\dmo{\triv}{triv}
\dmo{\genus}{genus}
\dmo{\Homeq}{HomEq}
\dmo{\PP}{\mathbb{P}}
\dmo{\U}{U}
\dmo{\Gal}{Gal}
\dmo{\BDiff}{\wtil{\Diff}}
\dmo{\BAut}{\wtil{\Aut}}
\dmo{\Iso}{Iso}
\dmo{\SL}{SL}
\dmo{\Cone}{Cone}
\dmo{\codim}{codim}
\dmo{\II}{II}
\dmo{\I}{I}
\dmo{\InjRad}{InjRad}
\dmo{\Inn}{Inn}
\dmo{\sys}{sys}
\dmo{\Comm}{Comm}
\dmo{\PO}{PO}
\dmo{\vertex}{Vert}
\dmo{\POm}{P\Om}
\dmo{\ab}{ab}
\dmo{\PSO}{PSO}
\dmo{\CRS}{CRS}
\dmo{\Tr}{Tr}

\newcommand{\mrt}{Millson--Raghunathan tuple}
\newcommand{\introthms}{\ref{thm:SL}--\ref{thm:ROS}}

\usepackage{pdfpages,stmaryrd}
\usepackage{enumitem}
\usepackage[margin=1.5in]{geometry}

\usepackage{tikz}
\usepackage{dsfont}
\usetikzlibrary{matrix}

\begin{document}

\title{Counting flat cycles in the homology of locally symmetric spaces}

\author{Daniel Studenmund and Bena Tshishiku}

\maketitle

\begin{abstract}
For a prime $p$, let $\Ga(p^\ell)<\SL_{n+1}(\Z)$ be a principal congruence subgroup and let $Y(p^\ell)$ denote its locally symmetric space. We prove that the subspace of $H_n(Y(p^\ell);\Q)$ generated by flat cycles grows at least as fast as as a linear function of $\vol(Y(p^\ell))^{\fr{n+1}{n^2+2n}}$. We also prove similar results for other locally symmetric spaces. This addresses a question of Avramidi--Nguyen-Phan.
\end{abstract}

\tableofcontents

\section{Introduction}\label{sec:introduction}

Let $G<\GL_d$ be an algebraic group defined over $\Q$ whose real points $G(\R)$ is a semisimple Lie group of noncompact type. It is a basic problem to understand the (co)homology of the arithmetic group $G(\Z)$ and its principal congruence subgroups
\[\Ga(s):=\ker\big[G(\Z)\ra G(\Z/s\Z)\big].\]

\begin{prob} \label{prob:growth}
For a fixed prime $p$ and $k\ge1$, determine the rate of growth of the homology groups $H_k(\Ga(p^\ell);\Q)$ as $\ell$ increases.
\end{prob}

Let $Y(p^\ell)=\Ga(p^\ell)\bs G(\R)/K$ denote the  locally symmetric space for $\Gamma(p^\ell)$. Then \[H_*(\Ga(p^\ell);\Q)\cong H_*(Y(p^\ell);\Q),\] and by the theory of $L^2$-Betti numbers, if $k\neq \frac{1}{2}\dim Y(p^\ell)$, then $H_k(Y(p^\ell);\Q)$ grows sublinearly in terms of the volume $\vol(Y(p^\ell))$, i.e.\ 
\begin{equation}\label{eqn:L2}\lim_{\ell\ra\infty}\fr{\dim H_k(Y(p^\ell);\Q)}{\vol(Y(p^\ell))}=0.\end{equation} 
See e.g.\ \cite[\S0.6,\S5.2]{lueck-book} and \cite{Gaboriau}. When $k=\frac{1}{2}\dim Y(p^\ell)$ then $H_k(Y(p^\ell);\Q)$ grows linearly in the volume for certain $G$ \cite[op.\ cit.]{lueck-book}.

No other general results are known about the growth of $H_k(Y(p^\ell);\Q)$ outside of Borel's stable range. Some sporadic results are known for $G(\R)=\SO(n,1)$ \cite{xue} and $G(\R)=U(n,m)$ \cite{marshall,marshall-shin,gerbelli-gauthier,stover}. 

Our main results provide lower bounds on the rate of growth of $H_n(Y(p^\ell);\Q)$ for certain $G$, where $n$ is the real rank of $G$. Precise statements follow.



\begin{mainthm}\label{thm:SL}
Fix $n\ge2$ and let $G=\SL_{n+1}$. Denoting $\Ga(s)<\SL_{n+1}(\Z)$ and $Y(s)=\Ga(s)\bs \SL_{n+1}(\R)/\SO(n)$ as above, 
there are infinitely many primes $p$ with the following property. There exist constants $\ca C, \ca L > 0$ such that 
\[\dim H_n(Y(p^\ell);\Q)\ge \ca C\cdot\vol(Y(p^\ell))^{\fr{n+1}{n^2+2n}}\]
for each $\ell\ge\ca L$. 
\end{mainthm}

\begin{mainthm}\label{thm:SO}
For $n=2m\ge2$, let $Q_n$ be the 
unimodular, even integral symmetric bilinear form with signature $(n,n)$. 
Let $G=\SO(Q_n)$ be its isometry group, and define $\Ga(s)<\SO(Q_n;\Z)$ and $Y(s)=\Ga(s)\bs G(\R)/K$ as above. 
For any $\ep > 0$, there are infinitely many primes $p$ with the following property. There exist constants $\ca C, \ca L > 0$ such that 
\[\dim H_n(Y(p^\ell);\Q)\ge\ca C\cdot \vol(Y(p^\ell))^{\fr{n-1-\ep}{n^2}}\]
for each $\ell\ge\ca L$. 
\end{mainthm}

Theorems \ref{thm:SL} and \ref{thm:SO} are proved by giving a lower bound on the dimension of the subspace of $H_n(Y(p^\ell);\Q)$ spanned by \emph{flat cycles}, i.e.\ compact, totally geodesic, flat submanifolds of $Y(p^\ell)$ of maximal dimension $n$. Using the same technique we can prove a similar result for Hilbert modular varieties. 

\begin{mainthm}\label{thm:ROS}
Fix $n\ge2$ and fix a degree-$n$ totally real number field $F/\Q$ with ring of integers $\ca O_F$. For a nonzero 
ideal $\mf s\sbs\ca O_F$, consider the congruence subgroup 
\[\Ga(\mf s)=\ker\big[\SL_{2}(\ca O_F)\ra\SL_2(\ca O_F/\mf s)\big]\] and let $Y(\mf s)=\Ga(\mf s)\bs (\bb H^2)^{\ti n}$. There are infinitely many prime ideas $\mf p\sbs\ca O_F$ with the following property. There exist constants $\ca C,\ca L>0$ such that if $\ell\ge\ca L$, then the subspace of $H_n(Y(\mf p^\ell);\Q)$ spanned by flat cycles has dimension at least $\ca C\cdot\vol(Y(\mf p^\ell))^{1/3}$.
\end{mainthm}

In Theorem \ref{thm:ROS}, the real rank $n$ is equal to half the dimension of $Y(\mf p^\ell)$, and it is known that $H_n(Y(\mf p^\ell);\Q)$ grows linearly in the volume, again by results in $L^2$-cohomology \cite[Thm.\ 5.12]{lueck-book} and \cite{Gaboriau}. In particular, the homology we produce using flat cycles only accounts for a fraction of $H_n(Y(\mf p^\ell);\Q)$. 

Avramidi--Nguyen-Phan \cite[\S9]{ANP} pose the problem of determining the growth rate of the subspace of homology spanned by flat cycles in congruence covers. Our results provide nontrivial lower bounds. We do not know of any upper bound other than that given by (\ref{eqn:L2}). In particular, it would be interesting to determine whether the exponent $\frac{1}{3}$ in Theorem \ref{thm:ROS} is sharp; if it is, then the growth of $H_n(Y(\mf p^\ell);\Q)$ is not fully accounted for by flat cycles. 


\begin{rmk}
Problem \ref{prob:growth} is related to a conjecture of Sarnak--Xue. Recall that the cohomology of an arithmetic group can be described representation-theoretically by Matsushima's formula, and computing $H_k(Y(p^\ell);\C)$ reduces to computing the multiplicity $m(\pi,p^\ell)$ of certain representations $\pi$ occurring in $L^2(\Ga(p^\ell)\bs G(\R))$. Sarnak--Xue \cite{sarnak-xue} conjectured upper bounds on the growth of these multiplicities. Ultimately, this has consequences for the growth of Betti numbers, but this conjecture has been approached in only a small number of cases. See, for example, the work of Marshall, Marshall--Shin, and Gerbelli-Gauthier \cite{marshall,marshall-shin,gerbelli-gauthier}, where the case $G=U(n,m)$ is studied using the theory of endoscopy. 
\end{rmk}

\paragraph{About the proof techniques.} 
This paper combines several ingredients, as indicated in the table of contents. The starting point is recent work \cite{ANP}, \cite{tshishiku-geocycles}, \cite{zschumme} that proves that flat cycles give nontrivial homology classes in some congruence cover $Y(p^k)$ in the cases considered in Theorems \ref{thm:SL}, \ref{thm:SO}, and \ref{thm:ROS}, respectively. We combine this with a simple topological argument of Xue \cite{xue} to estimate the number of linearly independent cycles in an orbit of the deck group of a further cover $Y(p^\ell)\rightarrow Y(p^k)$. This reduces the problem mostly to computing the size of the image of $H(\Z)\ra H(\Z/m\Z)$ for certain algebraic groups $H$. A key difficulty is that we must consider the case when $H$ is an algebraic torus, and these groups do not satisfy strong approximation. We overcome this with ad hoc arguments that use classical results from number theory about splitting primes in field extensions, including Chebotarev's density theorem and the Kronecker--Weber theorem; see e.g.\ Lemma \ref{lem:diagonalizable} and Proposition \ref{prop:splitting}. The orthogonal case (Theorem \ref{thm:SO}) is the most involved. 

Theorems \introthms\ remain true when their conclusions are strengthened from the existence of ``infinitely many primes'' to ``a set of primes of positive density.'' This is immediate from the proofs presented in \S\ref{sec:finite-lie} and \S\ref{sec:counting}, which require the exclusion of finitely many primes and applications Chebotarev's density theorem and Dirichlet's theorem on primes in arithmetic progressions.

\begin{rmk}
Examples of nonzero flat cycles in the homology of $Y=\Ga\bs G(\R)/K$ are constructed for $G=\SL_{n+1}$, $G=\SO(Q_n)$, and $G=R_{F/\Q}(\SL_2)$ by \cite{ANP}, \cite{tshishiku-geocycles}, and \cite{zschumme}, respectively. The same works \cite{ANP,tshishiku-geocycles,zschumme} prove that flat cycles span subspaces in $H_n(Y(p^\ell);\Q)$ of arbitrarily large dimension as $\ell$ increases.
Our main theorems both quantify and give an alternate proof of this latter fact, assuming that a single non-zero flat cycle exists (which is the most involved part of \cite{ANP, tshishiku-geocycles,zschumme}). 
\end{rmk}

\begin{rmk}
The cohomology of arithmetic groups has applications to the study of fiber bundles. In particular, Theorem \ref{thm:SO} implies a similar growth theorem for the homology for certain connected covers of the classifying space $B\Diff(W_n)$ for bundles with fiber $W_n=\#_n(S^{2d}\ti S^{2d})$ for $d\ge2$. For more information, see \cite[Cor.\ 2]{tshishiku-geocycles} and Theorem 29 in the Appendix of \cite{tshishiku-geocycles} by M.\ Krannich. 
\end{rmk}

\paragraph{Paper structure.} Sections \ref{sec:xue} and \ref{sec:geocycles} present the general approach used to prove each of Theorems \introthms. The details of the proofs appear in parallel in Sections \ref{sec:flats}--\ref{sec:counting}.

\paragraph{Acknowledgements.} The second author is supported by NSF grant DMS-2104346. The authors thank Jenny Wilson and Jeremy Miller for organizing of the Banff Workshop 21w5011, where the idea for this project was formed. Thanks to Kevin Schreve for a helpful email about L2 cohomology.

\section{Topological argument for homology growth in finite covers}\label{sec:xue}

The main goal of this section is to explain a simple topological argument of Xue \cite[\S3]{xue} that we use to prove our main theorems. Xue explains the argument in the context of arithmetic groups and congruence covers, but the argument applies more generally. We explain this below. It could be of independent interest. 

Fix an oriented manifold $Y$ (possibly noncompact) with oriented, embedded submanifolds $Y_1,Y_2\sbs Y$ of complementary dimension $n_1+n_2=\dim Y$. Assume that $Y_1$ is compact and that $Y_1,Y_2$ intersect transversely with nonzero algebraic intersection number $Y_1\cdot Y_2\neq0$. 

The submanifold $Y_1\sbs Y$ defines a homology class $[Y_1]\in H_{n_1}(Y;\Q)$. Since $Y_2$ is not assumed to be compact, it only defines a homology class in Borel--Moore homology with closed supports $H_{n_2}^{BM}(Y;\Q)$. By Poincar\'e duality, the intersection pairing between $H_{n_1}(Y;\Q)$ and $H_{n_2}^{BM}(Y;\Q)$ is a perfect pairing. Note in particular, that $[Y_1]\neq0$ in $H_{n_1}(Y;\Q)$ because $Y_1\cdot Y_2\neq0$. 

Fix a basepoint $y\in Y_1\cap Y_2$. Fix a finite quotient $\rho:\pi_1(Y)\onto W$, and let $\pi:\what Y\ra Y$ be the corresponding covering space. Choose a component $\what Y_i$ of $\pi^{-1}(Y_i)$ for $i=1,2$, and set $W_i=\rho(\pi_1(Y_i))$. In particular $\what Y_i\ra Y_i$ is a regular cover with deck group $W_i$. 

The following result gives a lower bound on how many $W$-translates of $\what Y_1$ are linearly independent in $H_{n_1}(\what Y;\Q)$. 

\begin{prop}[Xue]\label{prop:xue}
With the preceding setup, the $W$-translates of $[\what Y_1]\in H_{n_1}(\what Y;\Q)$ span a subspace of dimension at least $\frac{1}{|Y_1\cap Y_2|}\cdot\frac{|W|}{|W_1|\cdot|W_2|}$. 
\end{prop}

The proof is an easy consequence of the following linear algebra fact, proved in \cite[Lem.\ 6]{xue}. 

\begin{lem}\label{lem:linear-algebra}
Let $V$ be a vector space with dual $V^*$, and let $\phi\cdot u$ denote the pairing between $\phi\in V^*$ and $u\in V$. Let $S\sbs V$ and $S^*\sbs V^*$ be finite sets and assume that $(1)$ for each $\phi\in S^*$ there exists $u\in S$ such that $\phi\cdot u\neq0$ and $(2)$ there exists $t\ge1$ such that for each $u\in S$, there are at most $t$ elements in $S^*$ such that $\phi\cdot u\neq0$. Then 
\[\dim_\Q\text{\emph{span}}(S)\ge\frac{|S^*|}{t}\]
\end{lem}

\begin{proof}[Proof of Proposition \ref{prop:xue}]
First we bound the number of translates of $\what Y_2$ that have nontrivial algebraic intersection with $\what Y_1$. 
\[\begin{array}{rcl}
\#\{w(\what Y_2): \what Y_1\cdot w(\what Y_2)\neq0\}&\le& \#\{w(\what Y_2): \what Y_1\cap w(\what Y_2)\neq\vn\}\\[2mm]
&\le& |\what Y_1\cap\pi^{-1}(Y_2)|\\[2mm]
&=&|Y_1\cap Y_2|\cdot|W_1|.\end{array}\]
The first inequality holds because disjoint submanifolds have zero algebraic intersection number. 
The second inequality says that the number of components of $\pi^{-1}(Y_2)$ that intersect $\what Y_1$ is bounded by the total number of intersections (so there is equality precisely when each component of $\pi^{-1}(Y_2)$ intersects $\what Y_1$ at most once). For the last equality, observe that $\what Y_1\cap \pi^{-1}(Y_2)$ is the preimage of $Y_1\cap Y_2$ under the covering map $\what Y_1\ra Y_1$.

An identical bound holds for $\#\{w(\what Y_2): u(\what Y_1)\cdot w(\what Y_2)\neq0\}$ for any $u\in W$. 

The proposition now follows immediately by applying Lemma \ref{lem:linear-algebra} to $V=H_{n_1}(\what Y;\Q)$, $V^*=H_{n_2}^{BM}(\what Y;\Q)$, and $S$ (resp.\ $S^*$) equal to the $W$-translates of $[\what Y_1]$ (resp.\ $[\what Y_2]$). 
\end{proof}

\section{Algebraic groups and geometric cycles}\label{sec:geocycles}

Here we explain the general setup for studying geometric cycles in the homology of arithmetic locally symmetric manifolds, following methods of Millson \cite{millson} and Millson--Raghunathan \cite{millson-raghunathan}. A good reference is \cite{schwermer}. 

Let $G<\GL_d$ be an algebraic subgroup defined over $\Q$ such that $G(\R)$ is a semisimple Lie group. 
Fix reductive $\Q$-subgroups $G_1,G_2<G$. 

For any subring $R\subset\C$, we define  $G(R):=G\cap\GL_d(R)$, and similarly for $G_1,G_2$. If $G$ is actually defined over $\Z$ (as will hold for some of our examples), then for any ring $R$ (e.g.\ $R=\Z/t\Z$), we can define $G(R)=G\cap \GL_d(R)$. More generally, since $G$ is defined over $\Q$, the group $G(\Z/t\Z)$ is defined as long as $t$ is relatively prime to the (finitely many) denominators appearing in the defining equations of $G$.

For $r\ge1$, let $\Ga(r)<G(\Z)$ denote the congruence subgroup 
\[\Ga(r)=\ker\big[G(\Z)\hra\GL_d(\Z)\ra \GL_d(\Z/r\Z)\big],\]
and set $\Ga_i(r)=G_i(\Z)\cap\Ga(r)$. When $G$ is defined over $\Z$, we can equivalently define $\Ga(r)$ as the kernel of the reduction map $G(\Z)\ra G(\Z/r\Z)$. 

Suppose $K\sbs G(\R)$ is a maximal compact subgroup such that $K_i:=G_i(\R)\cap K$ is a maximal compact subgroup of $G_i(\R)$ for $i=1,2$. Then the orbit map $G_i(\R)\ra G(\R)/K$ defined by $g\mapsto gK$ descends to a totally geodesic embedding of $X_i:=G_i(\R)/K_i$ into $X:=G(\R)/K$. The embedding $X_i\hra X$ descends to an immersion of $Y_i(r):=\Ga_i(r)\bs G_i(\R)/K_i$ into $Y(r):=\Ga(r)\bs G(\R)/K$ for $i=1,2$.

\begin{defn}
Let $G\leq \GL_n$ and $G_1,G_2 \leq G$ be as above. For a maximal compact subgroup $K\sbs G(\R)$ and integer $r\geq 1$, say that $(G,G_1,G_2,K,r)$ is a {\em \mrt}\ if, using the above notation, the following conditions hold:
\begin{enumerate}[label=({A}\arabic*), ref={A}\arabic*]

\item \label{eqn:compact} $G_1(\Z)$ is cocompact in $G_1(\R)$.

\item \label{eqn:descends} $K_i \leq G_i(\R)$ is a maximal compact subgroup for $i=1,2$.

\item \label{eqn:dimension}
$\dim X_1+\dim X_2=\dim X$ and $X_1\cap X_2=\{eK\}$.

\item \label{eqn:embedded}
$Y_i(r)$ is orientable, and $Y_i(r)\ra Y(r)$ is an embedding.

\item \label{eqn:sign}
$Y_1(r),Y_2(r)$ intersect transversely and each intersection has the same sign.

\end{enumerate}
\end{defn}

Given a \mrt\ $(G,G_1,G_2,K,r)$, for any natural number $s$ that is a multiple of $r$, each component of the lift of $Y_1(r)$ over the covering map $Y(s) \to Y(r)$ represents a nonzero element of homology. In particular, following the results of \S\ref{sec:xue} we have the following bound on the dimension of homology groups of covers of locally symmetric spaces arising from \mrt s.

\begin{lem} \label{lem:mrtbound}
Suppose $(G,G_1,G_2,K,r)$ is a \mrt\ and let $d= \dim(X_1)$. Given $s$ any multiple of $r$, let  $W(r,s):=\Ga(r)/\Ga(s)$ and $W_i(r,s)=\Ga_i(r)/\Ga_i(s)$ for $i=1,2$. Then there is a constant $\ca C > 0$ such that 
\[
\dim( H_d(Y(s); \Q) ) \geq \ca C \cdot \frac{|W(r,s)|}{|W_1(r,s)|\cdot|W_2(r,s)|}
\]
\end{lem} 

\begin{proof}
As discussed in \S\ref{sec:xue}, condition (\ref{eqn:sign}) guarantees that the homology class $[Y_1(r)]$ in $H_d(Y(r);\Q)$ is nonzero, because it pairs nontrivially with the class $[Y_2(r)]$ in the Borel--Moore homology of $Y(r)$ with closed supports.

The group $W(r,s)$ is the deck group of the cover $Y(s)\ra Y(r)$. By Proposition \ref{prop:xue}, the number of linearly independent translates of $[Y_1(s)] \in H_d(Y(s);\Q)$ under $W(r,s)$ is at least 
\[\frac{1}{|Y_1(r)\cap Y_2(r)|}\cdot \frac{|W(r,s)|}{|W_1(r,s)|\cdot|W_2(r,s)|}\]
\end{proof}

\paragraph{Construction of \mrt s.} We prove Theorems \ref{thm:SL}--\ref{thm:ROS} by constructing \mrt s $(G,G_1,G_2,K,r)$ with suitable properties. While constructing \mrt s for general $G$ is challenging, many specific examples are known. We begin in \S\ref{sec:flats} by constructing $G_1<G$ as maximal $\R$-split tori $G_1(\R)\cong(\R^\ti)^n$ that are anisotropic over $\Q$.

Given $G$ and $G_1$, it is not always possible to choose $G_2$ so that condition (\ref{eqn:dimension}) holds. In \S\ref{sec:flat-homology} we build on the constructions of \S\ref{sec:flats} to explicitly construct $G_2 < G$ and $K<G(\R)$ satisfying (\ref{eqn:descends}) and (\ref{eqn:dimension}).

For $G$, $G_1$, and $G_2$ satisfying (\ref{eqn:compact})--(\ref{eqn:dimension}), condition (\ref{eqn:embedded}) holds quite generally. If $r$ is sufficiently large, then $Y_i(r)\ra Y(r)$ is an embedding by a result of Raghunathan; see \cite[Thm.\ E and its proof]{schwermer}. Similarly, one can find infinitely many primes $p$ so that if $r$ is a multiple of $p^k$ for $k$ sufficiently large, then $Y_i(r)$ is oriented for $i=1,2$; see \cite[Thm.\ F]{schwermer} and \cite[Prop.\ 2.2 and its proof]{rohlfs-schwermer}. When $G=\SL_{n+1}$, any prime $p$ admits such $k$, as argued in \cite[\S5]{ANP}. 

In special cases, it is known that one can choose $r$ so that condition (\ref{eqn:sign}) holds. This can be shown using an argument that originated in work of Millson \cite{millson} and Millson--Raghunathan \cite{millson-raghunathan}. We apply these techniques in \S\ref{sec:flat-homology}.

In light of Lemma \ref{lem:mrtbound}, to prove Theorems \ref{thm:SL}--\ref{thm:ROS} we are interested in bounding $\fr{|W(r,s)|}{|W_1(r,s)|\cdot|W_2(r,s)|}$ below. This is done in \S\ref{sec:counting}. We outline the approach here. By the definitions, there is a short exact sequence
\begin{equation}\label{eqn:congruence-SES}
1\ra W(r,s)\ra G(\Z)/\Ga(s)\ra G(\Z)/\Ga(r)\ra 1,
\end{equation}
and the group $G(\Z)/\Ga(t)$ is the image of $G(\Z)\hra\GL_d(\Z)\ra \GL_d(\Z/t\Z)$. We will provide lower bounds on the size of the image of $G(\Z)\ra\GL_d(\Z/s\Z)$, and upper bounds on the size of the images of $G_i(\Z)\ra\GL_d(\Z/s\Z)$ for $i=1,2$. For groups $G$ and $G_2$ we consider, these bounds will follow from strong approximation. Our groups $G_1$ are algebraic tori, for which strong approximation does not apply. 



\section{Explicit construction of periodic maximal flats}\label{sec:flats}

Let $X=G(\R)/K$ be a symmetric space of noncompact type, and let $\Ga<G(\R)$ be a lattice. A \emph{flat} is a geodesically embedded Euclidean space $\bb E^i\hra X$. The maximal dimension of a flat is the real rank of $G(\R)$, which we denote by $n$. A maximal flat is preserved by a Cartan subgroup $H$ of $G(\R)$, and we call the flat \emph{periodic} if $\Ga\cap H$ is a lattice in $H$. In this case, $(\Ga\cap H)\bs H$ is a compact flat manifold and there is an immersion $(\Ga\cap H)\bs H \looparrowright \Ga\bs X$. 

There are general results of Prasad--Raghunathan \cite{prasad-raghunathan} that guarantee the existence of periodic maximal flats. These results, however, are not constructive and are consequently difficult to apply in our setting. The explicit construction we give below will facilitate the computation in \S\ref{sec:finite-lie} for the proofs of Theorems \ref{thm:SL}--\ref{thm:ROS}. 




To construct periodic maximal flats in our cases, it suffices to find a $\Q$-subgroup $G_1<G$ such that $G_1(\R)\cong(\R^\ti)^n$ and $G_1(\Z)<G_1(\R)$ is cocompact (equivalently, $G_1(\Z)$ contains a free abelian group of rank $n$). For then $G_1(\R)$ has maximal compact subgroup $K_1\cong \{\pm1\}^n$ and $G_1(\R)/K_1\cong\bb E^n$ embeds in $G(\R)/K$ as a maximal flat for any maximal compact subgroup $K\sbs G_(\R)$, which is periodic because $G_1(\Z)<G_1(\R)$ is cocompact. 

\begin{rmk}For $G=\SL_{n+1}$, it is easy to find $G_1$. Fix a totally real number field $E/\Q$ of degree $n+1$, and let $G_1=R_{E/\Q}^1(G_m)$ be the norm torus in the restriction of scalars group $R_{E/\Q}(G_m)$. The group $G_1(\Z)\doteq\ca O_E$ is cocompact in $G_1(\R)\cong(\R^\ti)^{n}$ by Dirichlet's unit theorem. For the computation in \S\ref{sec:finite-lie}, it is helpful to assume further that the ring of integers $\ca O_E$ contains a unit $\la\in\ca O_E^\ti$ that is also a primitive element, i.e.\ $E=\Q(\la)$. It may be well known that such fields exist for each $n$, but we did not find a reference, so we give a construction below. We also note that this simple construction for $G=\SL_{n+1}$ does not seem to work as well for $G=\SO(Q_n)$. 
\end{rmk}

In each case, we define $G_1$ as the centralizer of a certain matrix $A$. There are two basic constructions that will be useful for showing that $G_1(\Z)<G_1(\R)$ is cocompact. Let $F/\Q$ be a number field with ring of integers $\ca O_F$.
\begin{itemize} 
\item Fix a polynomial $\xi\in\ca O_F[x]$ that is irreducible in $F[x]$, and let $\la$ be a root of $\xi$. View the number field $E=F[x]/(\xi)\cong F(\la)$ as an $F$ vector space with basis $1,\la,\ldots,\la^{d-1}$, where $d=\deg(\xi)$. Multiplication by $E$ on itself defines a ring homomorphism 
\[\rho:E\ra M_d(F).\] The matrix $\rho(\la)$ belongs to $M_d(\ca O_F)$ and has characteristic polynomial $\xi$ (this is the so-called \emph{companion matrix} of $\xi$). If $\xi(0)=1$, then $\rho(\la)\in\SL_d(\ca O_F)$. In this case $\ca O_F[\la]^\ti$ maps into the centralizer of $\rho(\la)$ in $\GL_d(\ca O_F)$. 
\item Let $A\in\SL_d(\ca O_F)$ be a matrix whose characteristic polynomial $\chi$ is irreducible in $F[x]$, and let $\la$ be a root of $\chi$ (i.e.\ an eigenvalue of $A$). The field $F[x]/(\chi)$ can be identified with 
\begin{enumerate}
\item[(1)] the subalgebra of $M_d(F)$ generated by $A$ (the map $F[x]\ra M_d(F)$ defined by $x\mapsto A$ factors through $(\chi)$ by the Cayley--Hamilton theorem, and this ideal is the kernel because $\chi$ is the minimal polynomial of $A$), and
\item[(2)] the vector space $V=F^d$ viewed as a $F[x]$-module with $x$ acting by $A$ (for any nonzero $v\in V$, the vectors $v,Av,\ldots,A^{d-1}v$ are linearly independent -- $A$ acts irreducibly on $V$ because $\chi$ is irreducible).
\end{enumerate}
\end{itemize} 

\subsection{Periodic maximal flats for $\SL_{n+1}$} \label{sec:flats-SL}

Fix a polynomial 
\[\xi=x^{n+1}+a_nx^n+\cdots+a_{1}x+1\]
with integer coefficients $a_i\in\Z$, and assume that $\xi$ is irreducible in $\Q[x]$ and has only real roots. Consider the companion matrix 
\[A=\left(\begin{array}{ccccc}
0&0&\cdots&0&-1\\
1&0&\cdots&0&-a_1\\
0&1&\cdots&0&-a_2\\
\vdots&\vdots&\ddots&\vdots&\vdots\\
0&0&\cdots&1&-a_n
\end{array}\right)\in\SL_{n+1}(\Z).\]

\begin{lem}\label{lem:SL-centralizer}
For $\xi$ and $A$ as above, define $G_1<\SL_{n+1}$ be the centralizer of $A$. Then $G_1(\R)\cong(\R^\ti)^n$, and $G_1(\Z)$ is cocompact in $G_1(\R)$. 
\end{lem}



\begin{proof}
Since $A$ has irreducible characteristic polynomial and real roots, $A$ is diagonalizable over $\R$. This implies that $G_1(\R)\cong(\R^\ti)^n$. 

To show $G_1(\Z)$ is cocompact in $G_1(\R)$, we show that $G_1(\Z)$ contains an abelian subgroup of rank $n$. Consider the subalgebra $\Q[A]\sbs M_{n+1}(\Q)$ generated by $A$. As mentioned above, $\Q[A]$ is isomorphic to the field $E=\Q[x]/(\xi)$. The group $\Z[A]^\ti$ is contained in $\GL_{n+1}(\Z)$ and centralizes $A$. The group $\Z[A]^\ti\cap\SL_{n+1}(\Z)$ has finite index in $\Z[A]^\ti$ and is contained in $G_1(\Z)$, so it suffices to show $\Z[A]^\ti$ has rank $n$. For this, note that $\Z[A]^\ti\cong\Z[x]/(\xi)$ has the same rank as $\ca O_E^\ti$, which is $n$ by Dirichlet's unit theorem; cf.\ \cite[Ch.\ I, Thms.\ 7.4 and 12.12]{neukirch}. 
%
\end{proof}

\begin{rmk}[Irreducible is necessary]
If $\xi$ is reducible and square-free with real roots, then again $G_1(\R)=(\R^\ti)^n$, but now $G_1(\Z)$ has rank $<n$ and hence is not cocompact. For example, suppose $A$ is a block matrix $B\oplus C$, where $B,C$ have irreducible characteristic polynomials which are relatively prime. Then $A$ can be diagonalized over $\R$ using a block diagonal matrix, and this implies that its centralizer is also block diagonal. Now one computes $\rank G_1(\Z)<n$ using Dirichlet's unit theorem similar to the proof of Lemma \ref{lem:SL-centralizer}. 
\end{rmk}


Next we show that for each $n\ge0$, there is a degree-$(n+1)$ monic polynomial $\xi\in\Z[x]$ that is irreducible in $\Q[x]$, has only real roots, and satisfies $\xi(0)=1$. 
There are likely several arguments for this, but we were not able to find a reference. Below we make a few remarks, and then give explicit examples in Lemma \ref{lem:irreducible-poly}. 

\begin{rmk}\label{rmk:rivin}
The characteristic polynomial of a ``random" element in $\SL_{n+1}(\Z)$ is irreducible by Rivin \cite{rivin}, but typically one expects the characteristic polynomial will have non-real roots. It would be nice to show that a random symmetric matrix in $\SL_{n+1}(\Z)$ has irreducible characteristic polynomial (since then the roots are necessarily real by the spectral theorem). There seem to be difficulties to adapting Rivin's argument \cite[Thm.\ 5.2]{rivin} to show this. 
\end{rmk}

\begin{rmk}
If $E/\Q$ is a number field and $\alpha\in\ca O_E^\ti$ is a unit and also $\alpha$ is a primitive element ($E=\Q(\al)$), then the minimal polynomial of $\al$ has the desired properties. However, not every number field has a primitive element that's also a unit \cite{primitive-unit}. 
\end{rmk}

\begin{rmk}
Kedlaya \cite{kedlaya} constructs for each $n=r+2s$, monic, irreducible polynomials $\xi\in\Z[x]$ so that the number field $F=\Q[x]/(\xi)$ has $r$ real embeddings and $s$ complex embeddings, and such that $\ca O_F=\Z[x]/(\xi)$. The polynomials $\xi$ are somewhat complicated, and we did not check if Kedlaya's construction includes examples with $\xi(0)=1$. In any case, the examples we give below are quite simple.
\end{rmk}

\begin{lem}\label{lem:irreducible-poly}
Fix $n\ge6$, and let $a_0,a_1,\ldots,a_n$ be integers such that 
\[0 = a_0<a_1<a_2<\cdots<a_n\>\>\>\text{ and }\>\>\>a_{i+1}-a_i > 2.\]
Then the degree-$(n+1)$ polynomial
\begin{equation}\label{eqn:irreducible-poly}\xi=1+x(x-a_1)\cdots(x-a_n) \in\Z[x]\end{equation}
has only real roots and is irreducible in $\Q[x]$. 
\end{lem}

Lemma \ref{lem:irreducible-poly} only gives examples for $n\ge6$, but it is easy to construct ad hoc examples for $n\le 5$. 
The proof of Lemma \ref{lem:irreducible-poly} uses the following irreducibility criterion, which can be found in \cite[Thm.\ 2.2.8]{prasolov}. 

\begin{thm}[Polya irreducibility criterion]\label{thm:polya}
Let $\xi$ be a degree-$d$ polynomial and let $k=\lfloor\fr{d+1}{2}\rfloor$. If there are distinct integers $b_0,\ldots,b_{d-1}$ such that 
\[0<|\xi(b_i)|<\frac{k!}{2^k}\>\>\>\text{ for }\>\>\>0\le i\le d-1,\]
then $\xi$ is irreducible in $\Q[x]$. 
\end{thm}

\begin{proof}[Proof of Lemma \ref{lem:irreducible-poly}]
We sketch the argument. To see that $\xi$ has $n+1$ real roots, let $m_0 = -2$, $m_i = \frac{a_{i-1}+a_{i}}{2}$ for $i=1,\dotsc, n$, and $m_{n+1} = a_n+2$. Then the assumption that $a_{i+1}-a_i > 2$ implies that the sign of $\xi(m_i)$ is $(-1)^{n-i}$, and so $\xi$ has $n+1$ real roots by the intermediate value theorem.

Next we explain why $\xi$ is irreducible. Observe that $\xi(0)=\xi(a_1)=\cdots=\xi(a_n)=1$. Since $1<\frac{k!}{2^k}$ for $k\ge4$, we can apply Polya's Theorem \ref{thm:polya} as long as $\lfloor\frac{n+2}{2}\rfloor\ge4$, i.e.\ $n\ge6$, to conclude that $\xi$ is irreducible. 
\end{proof}

\subsection{Periodic maximal flats for $R_{F/\Q}(\SL_2)$} \label{sec:flats-ROS}

This case can be treated similar to the $\SL_{n+1}$ case. Let $F/\Q$ be a totally real number field of degree $n\ge2$. 

Let $\xi=x^2+ax+1$ be a polynomial in $\ca O_F[x]$ that is irreducible in $F[x]$ and has real roots (equivalently, $a^2-4$ is positive and is not a square in $F$). Define 
\[A=\left(\begin{array}{cc}0&-1\\1&-a\end{array}\right)\in\SL_2(\ca O_F),\]
which has characteristic polynomial $\xi$. 

Consider $\ca H = \SL_2$ viewed as an algebraic group over $F$, and define $\ca H_1<\ca H$ to be the centralizer of $A$. Set $G=R_{F/\Q}(\ca H)$ and $G_1=R_{F/\Q}(\ca H_1)$. 




\begin{lem}\label{lem:ROS-centralizer}
The group $G_1(\R)$ is isomorphic to $(\R^\ti)^n$ and is embedded diagonally in $G(\R)=\big(\SL_2(\R)\big)^n$. Also, $G_1(\Z)<G_1(\R)$ is cocompact. \end{lem}

\begin{proof} The proof is similar to the proof of Lemma \ref{lem:SL-centralizer}. Since $A$ is diagonalizable in $\ca H_1(\R)$, its centralizer is isomorphic to $\R^\ti$. Applying restriction of scalars gives the first statement. 

For the second statement, we prove that $G_1(\Z)$ contains a rank-$n$ abelian group. The groups $G_1(\Z)$ and $\ca H_1(\ca O_F)$ are commensurable, so it suffices to work with $\ca H_1(\ca O_F)$.

Similar to the proof of Lemma \ref{lem:SL-centralizer}, consider the subalgebra $F[A]\sbs M_{2}(F)$ generated by $A$, which is isomorphic to the field $E=F[x]/(\xi)$. The group $(\ca O_F[A])^\ti$ is contained in $\GL_{2}(\ca O_F)$ and centralizes $A$. Consequently, $\ca H_1(\ca O_F)$ contains the kernel of the composition 
\[\big(\ca O_F[A])^\ti\ra\GL_2(\ca O_F)\xra{\det}\ca O_F^\ti.\]
This kernel has the same rank as the kernel of the norm map
\[\ca O_E^\ti\ra\GL_2(\ca O_F)\xra{\det}\ca O_F^\ti,\] 
which is $n=(2n-1)-(n-1)$ by Dirichlet's unit theorem. (Note that the norm map restricted to $\ca O_F^\ti$ is the map $\la\mapsto\la^2$, so the map $\ca O_E^\ti\ra\ca O_F^\ti$ is virtually surjective.) 
\end{proof}





\begin{rmk}[Example]
Consider $F=\Q(\sqrt{2})$ and $\xi=x^2-4x+1$. Then   
\[A=\left(\begin{array}{cc}0&-1\\1&4\end{array}\right)\in\SL_2(\ca O_F),\] 
which has $\la=2+\sqrt{3}$ as an eigenvalue. Here $\ca O_F=\Z\{1,\sqrt{2}\}$ and $E\cong\Q(\sqrt{2},\sqrt{3})$. The ring $\ca O_F[A]$ is isomorphic to 
\[\ca O_F[\la]=\Z\{1,\sqrt{2},\sqrt{3},\sqrt{6}\},\] which has finite index in $\ca O_E=\Z\{1,\sqrt{2},\sqrt{3},\frac{\sqrt{2}+\sqrt{6}}{2}\}$. We can see in this case that the centralizer $\ca H_1(\ca O_F)$ has rank 2 since $\la$ and $(5-4\sqrt{2})+(2\sqrt{2})\la$ belong to $\ca O_F[\la]^\ti$ and are linearly independent. The corresponding commuting matrices $A$ and $(5-4\sqrt{2})I+(2\sqrt{2})A$ in $\SL_2(\ca O_F)$ are
\[\left(\begin{array}{cc}0&-1\\1&4\end{array}\right)\>\>\>\text{ and }\>\>\>\left(\begin{array}{cc}
5 - 4 \sqrt{2}& -2\sqrt{2}\\
2 \sqrt{2}& 5 + 4\sqrt{2}
\end{array}\right).\]
\end{rmk}

\subsection{Periodic maximal flats in $\SO(Q_n)$}\label{sec:flats-SO}

Define 
\begin{equation}\label{eqn:hyperbolic-form}Q_n=\left(\begin{array}{cc}0&I_n\\I_n&0\end{array}\right)\end{equation}
where $I_n$ is the $n\ti n$ identity matrix. The bilinear form defined by $Q_n$ is the unique unimodular, even integral
symmetric bilinear 
form with signature $(n,n)$. For a ring $R$, we define
\[\SO(Q_n;R)=\{g\in\SL_{2n}(R): g^tQ_ng=Q_n\}.\]

Explicitly constructing periodic maximal flats in this case is more subtle than the preceding cases because there is no known simple characterization of the characteristic polynomials of elements in $\SO(Q_n)$. See e.g.\ \cite{gross-mcmullen} and \cite{bayer}. The main result we use follows.

\begin{lem}\label{lem:SO-centralizer}
Fix $n\ge2$, and fix a rational
symmetric bilinear 
form $Q$ of signature $(n,n)$. Assume that $A\in\SO(Q;\Z)$ has characteristic polynomial $\chi$ that's irreducible in $\Q[x]$ and has only real roots. Let $G_1<\SO(Q)$ be the centralizer of $A$.  Then $G_1(\R)\cong(\R^\ti)^n$ and $G_1(\Z)$ is cocompact in $G_1(\R)$. 
\end{lem}


\begin{rmk}
We do not actually know if there always exists $A\in\SO(Q;\Z)$ with irreducible characteristic polynomial and only real roots. Generically one expects that the characteristic polynomial is irreducible by an argument similar to \cite[Thm.\ 5.2]{rivin}. (For $Q=Q_n$, one can use \cite{gross-mcmullen} to prove that the map from (orthogonal) matrices to characteristic polynomials is dominant as a map of varieties.) But similar to Remark \ref{rmk:rivin}, we are unaware of an argument that ensures that there are examples with only real roots. 
\end{rmk}

\begin{rmk}
From the proof we will see that we do not need to assume that $\chi$ is irreducible. For example, we have the same conclusion if $\chi=\chi_1\cdots\chi_j$ is a product of irreducible polynomials with real roots. We will use this observation to construct explicit examples in each $\SO(Q_n)$ after the proof of Lemma \ref{lem:SO-centralizer}.
\end{rmk}  

\begin{rmk}
The statement of Lemma \ref{lem:SO-centralizer} can be generalized to $Q$ of signature $(n,m)$ with minor changes. For example, if $Q'=Q\op P$ where $Q$ has signature $(n,n)$ and $P$ is positive (or negative) definite, then a conclusion similar to that of Lemma \ref{lem:SO-centralizer} holds for $A\in \SO(Q')$ of the form $A=A'\oplus\id$, where $A'\in\SO(Q)$ has irreducible characteristic polynomial with real roots. In this case the characteristic polynomial is not irreducible 
and the centralizer of $A$ is not a Cartan subgroup, rather $G_1(\R)$ is the product of $(R^\ti)^n$ with the compact group $\SO(P;\R)$. 
\end{rmk}

\begin{proof}[Proof of Lemma \ref{lem:SO-centralizer}]
By assumption, the action of $A$ on $\R^{2n}$ has $2n$ distinct real eigenvalues, which occur in $(\la,1/\la)$ pairs with $\la\neq\pm1$. From this one deduces that there is an eigenbasis $u_1,\ldots,u_n,v_1,\ldots,v_n\in\R^{2n}$ consisting of isotropic vectors whose intersection matrix is $Q_n$. Then it is not hard to see that $G_1(\R)$ is isomorphic to $\SO(1,1)^n\cong(\R^\ti)^n$. 


Now we show that $G_1(\Z)$ is cocompact in $G_1(\R)$ by showing that $G_1(\Z)$ contains a rank-$n$ free abelian group. For this, we use some structural results of Milnor \cite{milnor-isometries}. Let $V\cong\Q^{2n}$ denote the standard representation of $\SO(Q;\Q)$. For any nonzero vector $v\in V$, the vectors $v,Av,\ldots,A^{2n-1}v$ are linearly independent (otherwise the action of $A$ on $V$ would be reducible, contradicting the fact that $\chi$ is irreducible). This allows us to identify $E:=\Q[x]/(\chi)$ with $V$ as $\Q[x]$-modules, with multiplication by $x$ on $\Q[x]/(\chi)$ corresponding to the action of $A$ on $V$. According to \cite[Lem.\ 1.1]{milnor-isometries}, there exists a unique 
bilinear 
form $h:E\ti E\ra E$ that is Hermitian with respect to the involution $e\mapsto \bar e$ on $E$ defined by $\bar x= x^{-1}$ and such that 
\[(u,v)\mapsto\Tr_{E/\Q}(h(u,v))\]
is equal to the given form $Q$ on $V$ (with respect to the identification of $E$ and $V$). From this we conclude that the action of $e\in E$ on $V$ preserves the form $Q$ if and only if $e\bar e=1$. 

Let $\la$ be a root of $\chi$. Under the map $E\cong\Q(\la)\ra M_{2n}(\Q)$ sending $\la$ to $A$, the subgroup of $\Z[\la]^\ti$ of elements satisfying $e\bar e=1$ maps into $G_1(\Z)$. We will show this subgroup has rank $n$. Since $\Z[\la]^\ti$ has finite index in $\ca O_E^\ti$, it suffices to prove the same statement for $\ca O_E^\ti$. 

Define $\Om^\pm=\{e\in \ca O_E^\ti: \bar e = e^{\pm 1}\}$, each of which is a subgroup of $\ca O_E^\ti$. We want to compute the rank of $\Om^-$. The product $\Om^+\ti\Om^-$ is a finite-index subgroup of $\ca O_E^\ti$ (this is a general fact about integral representations of $\Z/2\Z$). The group $\Om^+$ coincides with $\ca O_F$, where $F=\{e\in E:\bar e=e\}$ is the fixed field of the involution. Now $F/\Q$ has degree $n$ because $E/F$ has degree 2.
Then $\rank \Om^+=\rank \ca O_F^\ti=n-1$. Since $\rank\ca O_E^\ti=2n-1$, this implies that $\rank\Om^-=n$. This completes the proof. 
\end{proof}

Next we construct explicit examples $A\in \SO(Q_n)$ for every even $n\ge2$. Start with the case $n=2$ and consider the matrix  
\begin{equation}\label{eqn:orthogonal-matrix}
B=\left(\begin{array}{rrrr}
1& 1& -2& 2\\
1& 2& -4& 2\\
-2& -4& 10& -5\\
2& 2& -5& 5
\end{array}\right).
\end{equation}
This matrix preserves the 
bilinear 
form $Q_2$ and has characteristic polynomial
\begin{equation}\label{eqn:orthogonal-poly}\chi=1 - 18 x + 43 x^2 - 18 x^3 + x^4.\end{equation}
This polynomial has real roots (note $B$ is symmetric) and is irreducible in $\Q[x]$ (according to Mathematica).\footnote{The matrix $B$ was basically constructed as $U^tU$ with $U$ a random upper-triangular unipotent matrix in $\SO(Q_2)$.} 


Applying Lemma \ref{lem:SO-centralizer} to the matrix $B$ above, we obtain periodic maximal flats for $G(\Z)\bs G(\R)/K$ when $G=\SO(Q_2)$ and $K\sbs G(\R)$ is any maximal compact subgroup. (Note that $\SO(2,2)$ is isogenous to $\SL_2(\R)\ti\SL_2(\R)$, so this case is related to \S\ref{sec:flats-ROS}.)


We can obtain periodic maximal flats for any even $n\ge2$ as follows. Fixing $n=2m$, consider the block diagonal matrix 
\begin{equation} \label{eqn:Ablockdefn}
A=B\oplus B^2\oplus \ldots\oplus B^m,
\end{equation}
and let $G_1 < \SO(Q_n)$ be the centralizer of $A$. Then by the proof of Lemma \ref{lem:SO-centralizer}, we conclude that $G_1(\R)\cong(\R^\ti)^n$ and $G_1(\Z)$ is cocompact in $G_1(\R)$. A key point here is that $A$ has no repeated eigenvalues; otherwise its centralizer is larger than a Cartan subgroup and the symmetric space for $G_1(\R)$ is not flat.

The following proposition will be used in the proof of Theorem \ref{thm:SO} in \S\ref{sec:SO-proof}. 

\begin{prop}\label{prop:splitting}
For $\chi=1-18 x+43 x^2-18 x^3+x^4$ and $E=\Q[x]/(\chi)$, there are infinitely many primes $p\equiv3\pmod4$ that split completely in $E$. 
\end{prop}

\begin{proof}

First we note that $E=\Q[x]/(\chi)$ is a Galois extension of $\Q$ and has Galois group $\Z/2\Z\ti\Z/2\Z$. We checked this using Sage. 

By the Kronecker--Weber theorem \cite[Ch.\ V, Thm.\ 1.10]{neukirch}, every abelian extension of $\Q$ is contained in a cyclotomic field $\Q(\ze_m)$, where $\ze_m$ a primitive $m$-th root of unity. In fact, $E\sbs\Q(\ze_{40})$. To see this, first compute that the roots of $\chi$ are 
\[\begin{array}{c}
\fr{1}{2}\big(9 - 6\sqrt{2} - 3\sqrt{5} + 2\sqrt{10}\big)\\[2mm]
\fr{1}{2}\big(9 - 6\sqrt{2} + 3\sqrt{5} - 2\sqrt{10}\big)\\[2mm]
\fr{1}{2}\big(9 + 6\sqrt{2} - 3\sqrt{5} - 2\sqrt{10}\big)\\[2mm]
\fr{1}{2}\big(9 + 6\sqrt{2} + 3\sqrt{5} + 2\sqrt{10}\big)\end{array}\]
This implies that $E\cong\Q(\sqrt{2},\sqrt{5})$, and since
\[\sqrt{2}=\ze_8+\ze_8^{-1}\>\>\>\text{ and }\>\>\>\sqrt{5}=(\ze_5+\ze_5^{-1})-(\ze_5^2+\ze_5^{-2})\]
we conclude that $E\sbs\Q(\ze_5,\ze_8)\sbs\Q(\ze_{40})$.

The primes $p\nmid 40$ that split completely in $E$ are characterized as follows. The Galois group of $\Q(\ze_{40})/\Q$ is $(\Z/40\Z)^\ti$ and contains $\Gal(\Q(\ze_{40})/E)$ as a subgroup. If $p\nmid 40$ then $p$ splits completely if and only if the image of $p$ in $(\Z/40\Z)^\ti\cong\Gal(\Q(\ze_{40})/\Q)$ lies in $\Gal(\Q(\ze_{40})/E)$; \cite[Ch.\ I, \S10 and Ch.\ VII, Lem.\ 13.5]{neukirch}. One computes
\[\Gal(\Q(\ze_{40})/E)=\{1,9,31,39\}\sbs(\Z/40\Z)^\ti.\]
Therefore, $\chi$ splits completely in $\Z/p\Z$ whenever $p$ is a prime in one of the following arithmetic progressions. 
\[40x+1, \>\>\>40x+9,\>\>\> 40x+31,\>\>\> 40x+39\]
By Dirichlet's theorem, each of these arithmetic progressions contain infinitely many primes. In particular, since $40x+31\equiv3\pmod{4}$, there are infinitely many primes congruent to $3\pmod{4}$ so that $\chi$ splits completely in $\Z/p\Z$. 
\end{proof}

\section{Periodic maximal flats in homology}\label{sec:flat-homology}

In this section we complete our construction of \mrt s that will be used in the proofs of Theorems \ref{thm:SL}--\ref{thm:ROS}. For the algebraic tori $G_1$ constructed in \S\ref{sec:flats}, we construct $G_2$ and $K$ and identify primes $p$ with the property that $(G,G_1,G_2,K,p^k)$ is a \mrt\ for sufficiently large $k$. To do this, we apply the main results of \cite{ANP}, \cite{zschumme}, and \cite{tshishiku-geocycles} about flat cycles in homology.


\subsection{Flats and the homology of $\SL_{n+1}(\Z)$ }\label{sec:flat-homology-SL}


Let $G = \SL_{n+1}$. As in \S\ref{sec:flats-SL}, we define $G_1$ as the centralizer of $A_1\in\SL_{n+1}(\Z)$, where $A_1$ has real eigenvalues and whose characteristic polynomial is irreducible in $\Q[x]$. By Lemma \ref{lem:SL-centralizer}, the group $G_1$ is an $\R$-split torus $G_1(\R)\cong(\R^\ti)^n$, and $G_1(\Z)<G_1(\R)$ is cocompact. 

There is a unique maximal compact $K<G(\R)$ such that $K_1:=K\cap G_1(\R)$ is a maximal compact subgroup of $G_1(\R)$. Specifically, let $v_0,\ldots,v_n\in\R^{n+1}$ be an eigenbasis for $A_1$, fix the inner product for $\R^{n+1}$ where $v_0,\ldots,v_n$ is an orthonormal basis, and let $K\cong\SO(n+1)$ be the group of isometries of $\R^{n+1}$ with respect to this inner product. Then $K$ is a maximal compact subgroup of $G(\R)$ and it contains the maximal compact subgroup $K_1\cong(\Z/2\Z)^n$ of $G_1(\R)$.

We will choose $G_2$ as follows. Fix a vector $v\in\Q^{n+1}$, set $L=\R\{v\}$, and let $P$ be the orthogonal complement of $L$ in $\R^{n+1}$ (with respect to our fixed inner product). Let $A_2\in\GL_{n+1}(\Q)$ be the matrix that preserves the decomposition $\R^{n+1}=P\oplus L$, and acts on $P$ and $L$ by $-1$ and $+1$, respectively. Define $G_2$ to be the centralizer in $\SL_{n+1}$ of $A_2$. By construction, $K_2:=K\cap G_2(\R)$ is the maximal compact subgroup of $G_2(\R)$, and $K_2$ is isomorphic to $\Isom(P)\times\Isom(L)$.


The group $G_2(\Q)$ is conjugate in $\SL_{n+1}(\Q)$ to
\[\left\{\left(\begin{array}{cc}A&0\\0&\det(A)^{-1}\end{array}\right): A\in\GL_n(\Q)\right\}\sbs\SL_{n+1}(\Q).\]
We will impose further condition on the matrix $A_2$ below to ensure that $X_1$ and $X_2$ meet transversely in a single point. 

\paragraph{Intersection of $X_1$ and $X_2$.} 
Recall that $X_i:=G_i(\R)/K_i$, where $K_i=K\cap G_i(\R)$. 

We explain how to choose $A_2$ so that the $X_1$ and $X_2$ meet transversely in a single point. By construction, $X_1$ and $X_2$ intersect in the basepoint $eK\in G/K= X$, so the point is to show that the intersection is no larger. Avramidi--Nguyen-Phan \cite{ANP} give a projective geometry criterion for $X_1$ and $X_2$ to intersect transversely in a single point. Instead we use Lemma \ref{lem:generic-SL} below, which is simpler for our purposes.

Fix $v\in\Q^{n+1}$ such that if we write $v=a_0v_0+\cdots+a_nv_n$, then $a_i\neq0$ for each $i$. Set $L=\R\{v\}$ and $P$ be the orthogonal complement of $L$ in $\R^{n+1}$ (with respect to our chosen inner product). Let $A_2\in\GL_{n+1}(\Q)$ be the matrix that acts trivially on $L$ and by $-1$ on $P$. Observe that $K\cap G_2(\R)=\Isom(P)\ti\Isom(L)$ is a maximal compact subgroup of $G_2(\R)$. 

\begin{lem}\label{lem:generic-SL}
For $G_1$ and $G_2$ as above, $G_1(\R)\cap G_2(\R)=\{\pm I\}$. Consequently, $X_1\cap X_2=\{eK\}$. 
\end{lem}

\begin{proof}
Fix $g\in G_1(\R)\cap G_2(\R)$. Since $g$ commutes with $A_1$ and $A_2$, there exist nonzero constants $c_0,\ldots,c_n,c$ so that $g(v_i)=c_iv_i$ and $g(v)=cv$. 
On the one hand, 
\[g(v)=cv=c(a_0v_0+\cdots+a_nv_n)\]
and on the other hand, 
\[g(v)=g(a_1v_1+\cdots+a_nv_n)=a_0c_0v_0+\cdots+a_nc_nv_n.\]
Combining these we conclude that $c_i=c$ for each $i$. This shows that $g=cI$ is a scalar matrix, and $c\in\{\pm1\}$ because $g\in\SL_{n+1}(\R)$. 

To prove the last statement of the lemma, first observe that $X_1\cap X_2$ is connected (because $X_1$ and $X_2$ are totally geodesic in $X$), so if $X_1\cap X_2$ has more than one point, then $\dim\big(X_1\cap X_2\big)>0$. This implies that $\dim\big(G_1(\R)\cap G_2(\R)\big)>0$. 
\end{proof}

\paragraph{Intersection of $Y_1(r)$ and $Y_2(r)$.} Recall that $Y_i(r)=\Ga_i(r)\bs G_i(\R)/K_i$. Fix a prime $p$. The main theorem of \cite{ANP} proves the following. 

\begin{thm}[Avramidi--Nguyen-Phan]\label{thm:ANP}
Let $G,G_1,G_2,K$ be as above and fix a prime $p$. There exists $\ca K$ such that if $k\ge\ca K$, then $Y_1(p^k), Y_2(p^k)$ are orientable embedded submanifolds of $Y(p^k)$ that intersect transversely and such that every intersection has the same sign. 
\end{thm}

Combined with the general observations in \S\ref{sec:geocycles}, we conclude the following.
\begin{cor} \label{cor:mrt-SL}
Let $G,G_1,G_2,K$ be as above. For any prime $p$ there exists $\ca K$ such that if $k\ge\ca K$, then $(G,G_1,G_2,K,p^k)$ is a \mrt. 
\end{cor}

\subsection{Flats and the homology of $\SL_2(\ca O_F)$} \label{sec:flat-homology-ROS}

Fix a totally real number field $F/\Q$ of degree $n\ge2$. 
As in \S\ref{sec:flats-ROS}, consider $\ca H = \SL_2(\ca O_F)$ and $G = R_{F/\Q}(\ca H)$, and let $A_1\in\SL_2(\ca O_F)$ be the companion matrix of a polynomial $\xi=x^2+ax+1$ in $\ca O_F[x]$ that is irreducible in $F[x]$ and has real roots. Let $\ca H_1<\ca H$ be its centralizer and $G_1=R_{F/\Q}(\ca H_1)$. The group $G_1(\R)$ is isomorphic to $(\R^\ti)^n$ and is embedded in $G(\R)\cong(\SL_2(\R))^n$ diagonally by Lemma \ref{lem:ROS-centralizer}. The (unique) maximal compact subgroup $K_1<G_1(\R)$ maps to the subgroup $\{(\pm I,\ldots,\pm I)\}$ in $(\SL_2(\R))^n$. Since $\{\pm I\}$ is central in $\SL_2(\R)$, every maximal compact subgroup of $G(\R)$ contains $K_1$. As such, we fix any maximal compact subgroup $K<G(\R)$, and consider the corresponding maximal flat $X_1:=G_1(\R)/K_1\hookrightarrow G(\R)/K=:X$. 



\begin{lem}\label{lem:approximate-flat}
Let $A_1\in\SL_2(\ca O_F)$ as above. There exists $g\in\SL_2(F)$ such that the flats $X_1$ and $g(X_1)$ intersect transversely in a single point.  
\end{lem}

\begin{proof}
First we recall a few facts which make the lemma easy. 

{\it Parameterizing the space of flats.} Any flat $\ca F\sbs\hy^2\ti\cd\ti\hy^2$ is a product of geodesics $\ca F\cong\ga_1\ti\cd\ti\ga_n$. Since a geodesic in $\hy$ is determined uniquely by its endpoints in $\pa\hy^2\cong S^1$, a flat is uniquely determined by a tuple $(\theta_1,\ldots,\theta_{2n})\in(S^1)^{\ti2n}$. Since $\SL_2(\R)$ acts transitively on geodesics in $\hy^2$, it is obvious that $G(\R)$ acts transitively on flats. Furthermore, since $G(\Q)\cong\SL_2(F)$ is dense in $G(\R)$, the $G(\Q)$-orbit of a flat is dense in the space of flats (where the space of flats is topologized as a subset of $(S^1)^{\ti2n}$). 

{\it Intersection of two flats.} Two flats $\ca F$ and $\ca F'$ with corresponding tuples $(\theta_1,\ldots,\theta_{2n})$ and $(\theta_1',\ldots,\theta_{2n}')$ intersect transversely in a single point if and only if the pairs $(\theta_{2i-1},\theta_{2i})$ and $(\theta_{2i-1}',\theta_{2i}')$ are linked in $\pa\hy^2$ for each $i=1,\ldots,n$.

Now we prove the lemma. Choose a flat $\ca F$ that intersects $X_1$ transversely in a point. By the preceding paragraphs, we can choose $g\in\SL_2(F)$ so that $g(X_1)$ is close enough to $\ca F$ in the space of flats so that $g(X_1)$ and $X_1$ also intersect transversely in a point. 
\end{proof}


Using Lemma \ref{lem:approximate-flat}, choose $g\in\SL_2(F)$ so that $X_1$ and $X_2:=g(X_1)$ intersect transversely in a single point. Set $A_2=gA_1g^{-1}$, define $\ca H_2\cong g\ca H_1g^{-1}$ to be the centralizer of $A_2$, and define $G_2=R_{F/\Q}(\ca H_2)$. 


For an ideal $\mf r\sbs\ca O_F$, we consider the principal congruence subgroup
\[\Ga(\mf r):=\ker\big[\SL_2(\ca O_F)\ra\SL_2(\ca O_F/\mf r)\big].\]
Accordingly, we write $Y(\mf r)=\Ga(\mf r)\bs G(\R)/K$ and similarly define $\Ga_i(\mf r)$ and $Y_i(\mf r)$ for $i=1,2$. 

\begin{thm}[Zschumme]\label{thm:zschumme}
Let $G,G_1,G_2,K$ be as above. For any prime ideal $\mf p\sbs\ca O_F$ there exists $\ca K$ such that if $k\ge\ca K$, then $Y_1(\mf p^k),Y_2(\mf p^k)$ are oriented submanifolds of $Y(\mf p^k)$ that intersect transversely, and each intersection of $Y_1(\mf p^k)$ and $Y_2(\mf p^k)$ has the same sign. 
\end{thm}

Theorem \ref{thm:zschumme} has stricter hypotheses and a slightly stronger conclusion than what is proved by Zschumme \cite{zschumme}.  Zschumme considers more general lattices $R_{F/\Q}(\SL_D)$, where $D$ is a quaternion algebra, and only concludes that there is \emph{some} ideal $\mf r\sbs\ca O_F$ for which the conclusion above holds.

To reach the stronger conclusion in our case, note that a special feature of the $G=R_{F/\Q}(\SL_2)$ case is that $G_1(\R)$ automatically preserves the orientations on $X$ and $X_1$, and similarly for $G_2(\R)$. Therefore, once we choose $k$ large enough so that $Y_1(\mf p^k)$ and $Y_2(\mf p^k)$ are embedded, the conclusions about transversality and signs of the intersections are satisfied by general arguments. See \cite[\S2.2 and Prop.\ 6]{tshishiku-geocycles}.


Combined with the general observations in \S\ref{sec:geocycles}, we conclude the following.
\begin{cor} \label{cor:mrt-ROS}
Let $G,G_1,G_2,K$ be as above. For any prime $p$ there exists $\ca K$ such that if $k\ge\ca K$, then $(G,G_1,G_2,K,p^k)$ is a \mrt. 
\end{cor}

\subsection{Flats and the homology of $\SO(Q_n;\Z)$}\label{sec:flat-homology-SO}

Continue with the notation of \S\ref{sec:flats-SO}, so that $G = \SO(Q_n)$, and $G_1$ is the centralizer of the matrix $A\in\SO(Q_n;\Z)$ that was constructed in \S\ref{sec:flats-SO} equation \ref{eqn:Ablockdefn}. Then $G_1(\R)\cong(\R^\ti)^n$ and $G_1(\Z)<G_1(\R)$ is cocompact by Lemma \ref{lem:SO-centralizer}.

To define $G_2$, decompose $\R^{2n}=U_1\oplus\cdots\oplus U_n$ into 2-dimensional summands, where each $U_j$ is spanned by two by a $(\la,\frac{1}{\la})$-eigenspace pair for $A$. Fix a rational vector $v\in\R^{2n}$ such that $Q_n(v,v)<0$. Let $L$ be the span of $v$, and set $P=L^\perp$. Assume that $P\cap U_i$ is a positive line for each $i$. Define $G_2$ as the centralizer of $A_2\in\GL_{n+m}(\Q)$, where $A_2$ acts by $-1$ on $P$ and acts trivially on $L$. With this choice of $G_1,G_2$, there is a unique maximal compact subgroup $K\subset G(\R)$ such that $K_i=G_i(\R)\cap K$ is a maximal compact subgroup of $G_i(\R)$ for $i=1,2$ (maximal compact subgroups of $G(\R)$ correspond to maximal  positive-definite subspaces of $\R^{2n}$; here the unique choice for $K$ is the subspace spanned by the lines $P\cap U_i$). 




With these choices, $G_1\cap G_2=\{\id\}$ and $X_1,X_2\sbs X$ meet transversely in a single point, as explained in \cite[\S3]{tshishiku-geocycles}.  Furthermore, from \cite{tshishiku-geocycles} we have the following analogue of Theorem \ref{thm:ANP} of Avramidi--Nguyen-Phan \cite{ANP}. 

\begin{thm}\label{thm:tshishiku}
Let $G,G_1,G_2,K$ be as above. Let $p$ be prime and assume $p\equiv 3\pmod{4}$. There exists $\ca K$ so that if $k\ge\ca K$, then $Y_1(p^k), Y_2(p^k)$ are orientable embedded submanifolds of $Y(p^k)$ that intersect transversely and so that every intersection has the same sign. 
\end{thm}


Theorem \ref{thm:tshishiku} is a special case of what is proved in \cite[Thm.\ 1]{tshishiku-geocycles} and has a stronger conclusion. We explain this below. 

\begin{proof}[Proof sketch of Theorem \ref{thm:tshishiku}]
Using general theory, we can find $\ca K_0$ so that if $k\ge\ca K_0$, then $Y_1(p^k)$ and $Y_2(p^k)$ are oriented, embedded in $Y(p^k)$, and intersect transversely. Fixing $k\ge\ca K_0$, there is an embedding $\iota:Y_1(p^k)\cap Y_2(p^k)\hra\Ga_2(p^k)\bs\Ga(p^k)/\Ga_1(p^k)$. Fix $z\in Y_1(p^k)\cap Y_2(p^k)$ and write $[\ga]=\iota(z)$. To show that $Y_1(p^k)$ and $Y_2(p^k)$ intersect positively\footnote{Here we have fixed orientations on $X_1,X_2$ so that their intersection is positive.} at $z$ it suffices to show that we can write $\ga=g_2g_1$ where $g_i\in G_i(\R)$ preserves the orientation on $X_i$ and on $X$. As explained in \cite[Prop.\ 6]{tshishiku-geocycles}, it is possible to write $\ga=g_2g_1$, where $g_i\in G_i(\Q)$. To show that the orientations are preserved, we will show that the spinor norm $\theta(g_i)\in\R^\ti/(\R^\ti)^2\cong\{\pm1\}$ is trivial; cf.\ \cite[\S3.2]{tshishiku-geocycles}. (For more on the spinor norm, see e.g.\ \cite[\S7.2]{hahn-omeara}.) On $\SO(Q_n;\Z)$ the spinor norm factors through a surjection $\theta:\SO(Q_n;\Z)\ra\Z^\ti/(\Z^\ti)^2\cong\{\pm1\}$; therefore, $\eta\in\SO(Q_n;\Z)$ has trivial spinor norm if and only if it is in the kernel of the map
\begin{equation}\label{eqn:p-spinor}\SO(Q_n;\Z)\ra\SO(Q_n;\bb F_p)\ra \bb F_p^\ti/(\bb F_p^\ti)^2,\end{equation}
and in particular the spinor norm is trivial on $\Ga(p)$. The preceding statement uses the assumption $p\equiv3\pmod4$, which implies that the map in (\ref{eqn:p-spinor}) is surjective, since then $-1$ is not a square in $\bb F_p^\ti$. 


 Since $\ga\in\Ga(p)$, we know $\theta(\ga)=1$, and so $\theta(g_1)=\theta(g_2)$. To finish the proof, we will show that the spinor norm is trivial on $G_1(\Q)$. Over $\R$, a matrix $g\in G_1(\Q)$ can be diagonalized into the form $\left(\begin{array}{cc}D_g&0\\0&D_g^{-1}\end{array}\right)$, and the spinor norm is given by $\theta(g)=\det(D_g)\mod(\R^\ti)^2$ (as can be easily checked). It is easy to see that $\det(D_g)>0$ for each $g\in G_1(\Q)$ (the main observation here is that for $g\in G_1(\Q)$ the number of eigen-lines on which $g$ acts by $-1$ must be a multiple of 4; this is because the matrix $B$ used to define $A_1$ is a $4\ti 4$ matrix with irreducible characteristic polynomial, so the only diagonal matrix in $G_1(\Q)$ that commutes with $B$ is $-I$). 
 Therefore $\theta$ vanishes on $G_1(\Q)$. 
\end{proof}

Combined with the general observations in \S\ref{sec:geocycles}, we conclude the following.
\begin{cor} \label{cor:mrt-SO}
Let $G,G_1,G_2,K$ be as above. For any prime $p\equiv 3 \pmod{4}$, there exists $\ca K$ such that if $k\ge\ca K$, then $(G,G_1,G_2,K,p^k)$ is a \mrt. 
\end{cor}

\section{Size of Lie groups over finite rings}\label{sec:finite-lie}

The Lie groups of finite type that are of interest in this paper fall into two groups: (1) $\SL_d(\ca O_F/\mf p^j)$, where $F$ is a number field and $\mf p\sbs\ca O_F$ is a prime ideal, and (2) the orthogonal groups $O(Q;\Z/p^j\Z)$ for a prime number $p\neq2$. In this section we will provide estimates of the size of these groups as functions of $j$, considering all other variables fixed.

Given two functions $f,g$ of $j$, we write $f\sim g$ if there is a constant $C$ (independent of $j$) such that $\frac{1}{C}\cdot g(j)\le f(j)\le C\cdot g(j)$ for all $j\geq 1$.  

%

\subsection{Size of $\SL_d(\ca O_F/\mf p^j)$} 

Fix a number field $F/\Q$. Throughout this section we consider an ideal $\mf t\sbs\ca O_F$ of the form $\mf t=\mf p^j$ for some prime ideal $\mf p$. Set $\bb F=\ca O_F/\mf p$. 




\begin{prop}\label{prop:SL-size}
Fix $j\ge1$, and let $\mf t=\mf p^j$, where $\mf p\sbs\ca O_F$ is a prime ideal. For $d\ge2$, 
\[|\SL_d(\ca O_F/\mf t)|\sim N(\mf t)^{d^2-1},\]
where $N(\mf t)=|\ca O_F/\mf t|$ is the ideal norm. 
\end{prop} 

Proposition \ref{prop:SL-size} is surely well known. We include a proof for completeness and also because it is a good warm-up for the orthogonal case in the next section.

\begin{proof}[Proof of Proposition \ref{prop:SL-size}]
Set $R=\ca O_F/\mf t$. Given the short exact sequence
\[1\ra\SL_d(R)\ra\GL_d(R)\ra R^\ti\ra1,\]
we may break up our computation into computing $|\GL_d(R)|$ and $|R^\ti|$. 

\paragraph{Computing $|R^\ti|$.} Let $M\sbs R$ be the unique maximal ideal, which is the image of $\mf p\sbs\ca O_F$ under $\ca O_F\ra\ca O_F/\mf p^j\equiv R$. Note that $R/M\cong\bb F$. 

{\it Claim $1$.} The natural map $R^\ti\ra (R/M)^\ti=\bb F^\ti$ gives a short exact sequence
\begin{equation}\label{eqn:ring-units}1\ra(1+M)\ra R^\ti\ra \bb F^\ti\ra1.\end{equation}
Given the claim, we conclude that 
\begin{equation}\label{eqn:unit-size}|R^\ti|=|M|\cdot|\bb F^\ti|=N(\mf t)\>\fr{|\bb F^\ti|}{|\bb F|}.\end{equation}
Here we use that $|M|=\fr{|R|}{|R/M|}$ and $R/M=\bb F$. 

{\it Proof of Claim $1$.} We need to show that $R^\ti\ra (R/M)^\ti$ is surjective and has kernel $1+M$.

The kernel of $R^\ti\ra(R/M)^\ti$ is equal to $(1+M)\cap R^\ti$. In fact, $1+M$ is contained in $R^\ti$: for each $m\in M$, the element $1+m$ is invertible in $R$ because the ideal $M\sbs R$ is nilpotent ($M^j=0$). Therefore, $1+M$ is the kernel of $R^\ti\ra(R/M)^\ti$. 

To show that $R^\ti\ra(R/M)^\ti$ is surjective, fix $x\in R$ whose image in $R/M$ is invertible. To prove surjectivity, it suffices to show that $x$ is invertible. By assumption there exists $y\in R$ such that $xy=1+m$ for some $m\in M$. Since $1+m$ is invertible, this implies $xy$ is invertible; hence $x$ is also invertible, as desired.

\paragraph{Computing $|\GL_d(R)|$.} Similar to the preceding computation, we reduce to the following claim.

{\it Claim $2$.} The natural map $\GL_d(R)\ra\GL_d(R/M)=\GL_d(\bb F)$ gives a short exact sequence
\[1\ra \big[I+M_d(M)\big]\ra\GL_d(R)\ra\GL_d(\bb F)\ra1,\]
where $I$ is the $2\ti2$ identity matrix and $M_d(M)$ is the group of $d\ti d$ matrices with entries in $M$. 

The claim implies that
\begin{equation}\label{eqn:gl-size}|\GL_d(R)|=|M_d(M)|\cdot|\GL_d(\bb F)|=N(\mf t)^{d^2}\>\fr{|\GL_d(\bb F)|}{|\bb F|^{d^2}}. 
\end{equation}
Combining Equations (\ref{eqn:unit-size}) and (\ref{eqn:gl-size}), we conclude that 
\[|\SL_d(R)|=\fr{|\GL_d(R)|}{|R^\ti|}=N(\mf t)^{d^2-1}\cdot\fr{|\GL_d(\bb F)|}{|\bb F^\ti|\cdot|\bb F|^{d^2-1}},\]
which proves the proposition. Therefore, it remains only to prove Claim 2.

{\it Proof of Claim $2$.} The kernel of $\GL_d(R)\ra\GL_d(R/M)$ is the intersection of $I+M_d(M)$ with $\GL_d(R)$. But in fact $I+M_d(M)$ is contained in $\GL_d(R)$ because each $\mu\in M_d(M)$ is nilpotent ($\mu^j=0$). Furthermore, $\GL_d(R)\ra\GL_d(R/M)$ is surjective: given $A\in M_d(R)$ whose image in $M_d(R/M)$ is invertible, there exists $B$ so that $AB=I+\mu$ for some $\mu\in M_d(M)$. Since $1+\mu$ is invertible, this implies $AB$ and hence $A$ is invertible. 

This finishes the proof of Claim 2 and the proof of the proposition. 
\end{proof}

\subsection{Size of finite orthogonal groups}

Fix a prime $p\neq2$ and $d\ge2$. Let $Q\in M_d(\Z)$ be the matrix of a nondegenerate symmetric bilinear form $\pair{\cdot,\cdot}:\Z^d\ti\Z^d\ra\Z$, and assume that $Q$ remains nondegenerate after reducing mod $p$, i.e.\ $\det(Q)\not\equiv0\pmod{p}$. Define the orthogonal group
\[O(Q;\Z/p^j\Z)=\{g\in\GL_d(\Z/p^j\Z): g^t Q g=Q\}.\]
The reduction map 
\[\rho_j:\Z/p^j\Z\ra\Z/p^{j-1}\Z\] induces a homomorphism 
\[\pi_j:O(Q;\Z/p^j\Z)\ra O(Q;\Z/p^{j-1}\Z),\]
which we will use to inductively compute $|O(Q;\Z/p^j\Z)|$. For each $i\ge2$, set $M_i:=\ker(\rho_i)$, and define
\[S_i(Q)=\{\mu\in M_d(M_i): \mu^tQ+Q\mu=0\}.\]

\begin{prop}\label{prop:orthogonal-size}
With the preceding setup, 
\[|O(Q;\Z/p^j\Z)|=|O(Q;\Z/p\Z)|\cdot|S_2(Q)|\cdots|S_j(Q)|.\]
\end{prop}

\begin{proof}
First we show that $\pi_j$ is surjective for each $j\ge2$. For this, we use that $O(Q;\Z/p^j\Z)$ is generated by reflections 
\[r_x: v\mapsto v-2\frac{\pair{v,x}}{\pair{x,x}}\>x,\] where $x\in(\Z/p^j\Z)^d$ is such that $\pair{x,x}$ is a unit in $\Z/p^j\Z$; see \cite[Ch.\ I, Cor.\ 4.3]{milnor-husemoller}. Given this, it is straightforward to check that if $x\in(\Z/p^{j-1}\Z)^d$ is such that $\pair{x,x}$ is a unit and $\hat x\in (\Z/p^j\Z)^d$ is a lift of $x$, then $\pair{\hat x,\hat x}$ is a unit in $\Z/p^j\Z$, and $\pi_j(r_{\hat x})=r_x$. Thus $\pi_j$ is surjective. 

It remains to understand $\ker(\pi_j)$. Similar to the proof of Proposition \ref{prop:SL-size}, there is a short exact sequence 
\[1\ra\big[I+M_d(M_j)\big]\ra\GL_d(\Z/p^j\Z)\ra\GL_d(\Z/p^{j-1}\Z)\ra1,\]
so $\ker(\pi_j)$ is the intersection of $I+M_d(M_j)$ with $O(Q;\Z/p^j\Z)$. If $I+\mu$ belongs to $I+M_d(M_j)$ and $O(Q;\Z/p^j\Z)$, then $(1+\mu)^tQ(1+\mu)=Q$, which implies that $\mu^tQ+Q\mu=0$ (note that $\mu^tQ\mu$ is 0 because $M_j^2=0$ in $\Z/p^j\Z$). Therefore $|\ker(\pi_j)|=|S_j(Q)|$, and the Proposition follows by induction on $j$. 
\end{proof}

\paragraph{Example: $Q=Q_n$.} When $Q=Q_n$ as in \S\ref{sec:flats-SO}, one computes 
\[|S_i(Q_n)|=|M_i|^{n^2}+2|M_i|^{n(n-1)/2} = p^{n^2} + 2 p^{n(n-1)/2}.\]
for each $i\ge2$. 
Then Proposition \ref{prop:orthogonal-size} implies that if $p\neq2$ and we set $t=p^j$, then $|O(Q_n;\Z/t\Z)|\sim ( p^{n^2} + 2 p^{n(n-1)/2})^j$. 

For a ring $R$, define $\Om(Q_n;R)$ as the kernel of the spinor homomorphism
\begin{equation}\label{eqn:omega}\Om(Q_n;R):=\ker\big[\SO(Q_n;R)\xra{\theta} R^\ti/(R^\ti)^2\big].\end{equation}
When $R=\Z/p^j\Z$, the subgroups $\Om(Q_n;R)\sbs\SO(Q_n;R)$ and $\SO(Q_n;R)\sbs O(q_n;R)$ each have index at most 2, so we conclude the following corollary. 

\begin{cor}\label{cor:omega}
Fix a prime $p\neq2$, and set $t=p^j$ for $j\ge1$. Then
\[|\Om(Q_n;\Z/t\Z)|\sim \left( p^{n^2} + 2 p^{n(n-1)/2}\right)^j.\]
\end{cor} 

\paragraph{Example: $Q=\wtil Q_n$.} Define 
\begin{equation}\label{eqn:form2}\wtil Q_n=\left(\begin{array}{cc}Q_n&0\\0&1\end{array}\right)\in M_{2n+1}(\Z).\end{equation}
Here one computes that $|S_i(\wtil Q_n)|=|S_i(Q_n)|+|M_i|^n = p^{n^2} + 2 p^{n(n-1)/2} + p^n$ for each $i\ge 2$.

\begin{cor}\label{cor:omega2}
Fix a prime $p\neq2$, and set $t=p^j$ for $j\ge 1$. Then
\[|\Om(\wtil Q_n;\Z/t\Z)|\sim \left( p^{n^2} + 2 p^{n(n-1)/2} + p^n \right)^j.\]
\end{cor}

\subsection{Finite tori}


For the proof of the main theorems, we will need to know the size of the centralizer of certain elements in certain finite groups of Lie type. 

\begin{lem}\label{lem:diagonalizable}
Let $F/\Q$ be a number field. Fix $A\in\SL_d(\ca O_F)$ and assume that its characteristic polynomial $\chi$ is irreducible in $F[x]$. There are infinitely many prime ideals $\mf p\sbs\ca O_F$ such that for each $j\ge1$, 
\begin{enumerate}
\item[(i)] $\chi$ is square-free and splits completely in $(\ca O_F/\mf p^j)[x]$, and 
\item[(ii)] $A$ is diagonalizable in $\GL_{d}(\ca O_F/\mf p^j)$.
\end{enumerate} 
Consequently, for these primes, the centralizer of $A$ in $\GL_{d}(\ca O_F/\mf p^j)$ has size $\ca C^{d}$, where $\ca C=N(\mf p^j)\>\fr{N(\mf p)-1}{N(\mf p)}$ and $N(\mf t)=|\ca O_F/\mf t|$ is the ideal norm.
\end{lem}


For the last statement, compare with (\ref{eqn:unit-size}). As we will explain, (i) is well-known, and (i) implies (ii). 

\begin{rmk}There is reason to be cautious when trying to diagonalize a matrix $A\in \GL_d(R)$ when $R$ is a commutative ring that is not a field because in this generality, (1) a given eigenspace $\{v\in R^d: A v=\la v\}$ need not have a basis, (2) the sum of eigenspaces need not be a direct sum, and (3) the sum of eigenspaces need not span $R^d$. These phenomena are illustrated in \cite{richter-wardlaw}. These issues can be avoided if the characteristic polynomial of $A$ is square-free and splits completely in $\big(\ca O_F/\mf p\big)[x]$, as will be explained in the proof of Lemma \ref{lem:diagonalizable}. \end{rmk}

\begin{proof}[Proof of Lemma \ref{lem:diagonalizable}]

We begin with (i). Fix a root $\la$ of $\chi$, and let $E=F(\la)\cong F[x]/(\chi)$. 

First we recall that there are infinitely many primes $\mf p\sbs\ca O_F$ so that $\chi$ is square-free and splits completely in $\big(\ca O_F/\mf p\big)[x]$. By Chebotarev's density theorem \cite[Ch.\ VII, Cor.\ 13.6 and Ch.\ I, \S8, Exer.\ 4]{neukirch}, there are infinitely many primes $\mf p$ that split completely $\mf p\ca O_E=\mf q_1\cdots\mf q_d$. If $\mf p$ is such a prime, and if we assume further that $\mf p$ is not one of the (finitely many) divisors of the conductor ideal $\mf c\sbs\ca O_F$ of the ring $\ca O_F[\la]$, then $\chi$ is square-free and splits completely in $\big(\ca O_F/\mf p\big)[x]$; see \cite[Ch.\ I, Prop.\ 8.3]{neukirch}. 


If $\chi$ is square-free and splits completely in $\big(\ca O_F/\mf p\big)[x]$, then the same also holds in $\big(\ca O_F/\mf p^j\big)[x]$ for each $j\ge1$ by Hensel's lemma \cite[Ch.\ II, Lem.\ 4.6]{neukirch}. Furthermore, the uniqueness part of Hensel's lemma implies that $\chi$ has exactly $n$ roots in $\ca O_F/\mf p^j$, one lying over each root in $\ca O_F/\mf p$. This proves (i). 

Next we prove (ii). Fixing $j\ge 1$, we argue that $A$ is diagonalizable in $\GL_d(\ca O_F/\mf p^j)$ by showing $A$ has a basis of eigenvectors. Let $\la_1,\ldots,\la_d\in \ca O_F/\mf p^j$ be the roots of $\chi$, and let $E_1,\ldots,E_d\sbs (\ca O_F/\mf p^j)^n$ be the corresponding eigenspaces. Each $E_i$ contains a linearly independent vector $v_i$ \cite[Lem.\ 2]{richter-wardlaw} (here it is important that $\chi(\la_i)$ is zero and not merely a zero-divisor). Furthermore, the standard proof from linear algebra shows that $v_1,\ldots,v_n$ are linearly independent; cf.\ \cite[Lem.\ 3]{richter-wardlaw} (here it is important that for $i\neq j$, $\la_i-\la_j$ is not a zero-divisor in $\ca O_F/\mf p^j$, which holds because $\la_i$ and $\la_j$ map to distinct elements in $\ca O_F/\mf p$). Finally, $v_1,\ldots,v_n$ span $(\ca O_F/\mf p^j)^n$, as can be seen by counting since the span of $v_1,\ldots,v_n$ has the same cardinality as $(\ca O_F/\mf p^j)^n$. This shows that $v_1,\ldots,v_n$ are an eigenbasis for $A$, so $A$ is diagonalizable in $\GL_d(\ca O_F/\mf p^j)$. 
\end{proof}

We prove a similar theorem for the orthogonal case. 

\begin{lem}\label{lem:diagonalizable-SO}
Fix $A\in\SO(Q_n;\Z)$ and assume that its characteristic polynomial factors $\chi=\chi_1\cdots\chi_m$, where the $\chi_i$ are distinct and irreducible in $\Q[x]$. There exist infinitely many primes $p$ such that for each $j\ge1$, 
\begin{enumerate}
\item[(i)] $\chi$ is square-free and splits completely in $(\Z/p^j\Z)[x]$, and 
\item[(ii)] $A$ is diagonalizable in $\SO(Q_n;\Z/p^j\Z)$.
\end{enumerate} 
Consequently, for these primes, the centralizer of $A$ in $\SO(Q_n;\Z/p^j\Z)$ has size $\ca C^n$, where $\ca C=p^j\left(\fr{p-1}{p}\right)$.
\end{lem}

The proof is similar to the proof of Lemma \ref{lem:diagonalizable}, but with a few key differences. Specifically we need to deal with the fact that $\chi$ is not assumed to be irreducible and we need to show that $A$ is diagonalizable in $\SO(Q_n;\Z/p^j\Z)$ rather than just $\GL_{2n}(\Z/p^j\Z)$. 

\begin{proof}[Proof of Lemma \ref{lem:diagonalizable-SO}]

\boxed{(i)} Let $\la_i$ be a root of $\chi_i$, and set $E_i=F(\la_i)\cong F[x]/(\chi_i)$ for $1\le i\le m$. Set $\wtil E=F(\la_1,\ldots,\la_m)$. By Chebotarev's theorem, there are infinitely many primes that split completely in $\wtil E$, and hence also in each $E_i$ (see \cite[\S13]{neukirch}). 
As in the proof of Lemma \ref{lem:diagonalizable}, for such a prime $p$, if $p$ does not divide the conductor $c_i\in\Z$ of $\Z[\la_i]$, then $\chi_i$ splits completely in $(\Z/p^j\Z)[x]$ for each $j\ge1$. To finish the proof of (i), we argue that after imposing finitely many more constraints on $p$, we may assume that $\chi$ is squarefree in $(\Z/p^j\Z)[x]$. By assumption, for each $i<i'$, $\chi_{i}$ and $\chi_{i'}$ are relatively prime in $\Q[x]$, so 
\begin{equation}\label{eqn:rel-prime}a_{i,i'}\chi_i+b_{i,i'}\chi_{i'}=1\end{equation} for some $a_{i,i'},b_{i,i'}\in\Q[x]$. If $p$ does not divide the numerator or denominator of any coefficient of $a_{i,i'}$ or $b_{i,i'}$ (for $1\le i<i'\le m$), then the relation (\ref{eqn:rel-prime}) holds in $(\Z/p^j\Z)[x]$  (with both $a_{i,i'}$ and $b_{i,i'}$ nonzero), so $\chi_i,\chi_{i'}$ are relatively prime in $(\Z/p^j\Z)[x]$ (for each $i<i'$), which implies that they have no common roots. Hence $\chi$ is square-free. 

\boxed{(ii)} Fix a prime $p$ as in the preceding paragraph and fix $j\ge1$. By Lemma \ref{lem:diagonalizable}, $A$ is diagonalizable in $\GL_{2n}(\Z/p^j\Z)$. Let $v_1,\ldots,v_{2n}$ be an eigenbasis for $A$. 

We choose $p$ so that each eigenvector is isotropic. If $v$ is not isotropic, then its eigenvalue is $\pm1\pmod{p}$ (since $\pair{v,v}=\pair{Av,Av}=\la^2\pair{v,v}$). Therefore, if we assume $p$ does not divide $\chi(1)$ or $\chi(-1)$, then $\pm1$ are not roots of $\chi$ in $\Z/p^j\Z$, which implies $v_1,\ldots,v_{2n}$ are isotropic in $(\Z/p^j\Z)^{2n}$. 

Vectors in our eigenbasis come in pairs whose eigenvalues are inverses, so we write instead $v_1,\ldots,v_n,u_1,\ldots,u_n$, where $v_i,u_i$ have eigenvalues $\la_i,\frac{1}{\la_i}$. For each $i$, the vectors $v_i,u_i$ span a hyperbolic summand; furthermore, up to scaling $v_i$, we can assume that the intersection matrix of $v_i,u_i$ is 
\[\left(\begin{array}{cc}0&1\\1&0\end{array}\right).\]
Therefore there is an isometry of $(\Z/p^j\Z)^{2n}$ taking the standard basis $\{e_i,f_i\}$ to  $\{v_i,u_i\}$. This shows that $A$ is diagonalizable in $\SO(Q_n;\Z/p^j\Z)$. 
\end{proof}

%
%
%
%


\section{Counting flat cycles}\label{sec:counting}

In this section we complete the proofs of Theorem \ref{thm:SL}--\ref{thm:ROS}.

\subsection{$\SL_{n+1}$ case (Theorem \ref{thm:SL})} \label{sec:SL}

Fix $G,G_1,G_2$ as in \S\ref{sec:flat-homology-SL}. For any prime $p$, let $r=p^k$ for some fixed $k\gg0$  large enough that 
Corollary \ref{cor:mrt-SL} applies. For any $\ell>k$, let $s=p^\ell$. The main remaining task is to estimate the sizes of the groups $W(r,s)=\Ga(r)/\Ga(s)$ and $W_i(r,s)=\Ga_i(r)/\Ga_i(s)$ for $i=1,2$ as functions of $\ell$ for infinitely many primes $p$.

\paragraph{Size of $W(r,s)$.} We show that $|W(r,s)|\sim (p^{n^2+2n} )^\ell$.
Recall the following short exact sequence (\ref{eqn:congruence-SES}) from \S\ref{sec:geocycles}.
\[1\ra W(r,s)\ra G(\Z)/\Ga(s)\ra G(\Z)/\Ga(r)\ra1,\]
where $G(\Z)/\Ga(t)$ is isomorphic to the image of 
\[\SL_{n+1}(\Z)\ra\SL_{n+1}(\Z/t\Z).\] 
Because $r$ is fixed, it follows that $|W(r,s)| \sim |\SL_{n+1}(\Z/s\Z)|$.

It is well-known that $\SL_{n+1}(\Z)\ra\SL_{n+1}(\Z/t\Z)$ is surjective for each $t$, e.g.\ because these groups are generated by elementary matrices \cite[Thm.\ 4.3.9]{hahn-omeara}. Since $|\SL_{n+1}(\Z/t\Z)|\sim t^{(n+1)^2-1}=t^{n^2+2n}$, we conclude that $|W(r,s)|\sim \left(p^{n^2+2n}\right)^\ell$. 

\paragraph{Size of $W_2(r,s)$.} We show that $|W_2(r,s)|\sim \left(p^{n^2-1}\right)^\ell$ 
for all but finitely many $p$. 

Recall that $G_2(\Q)$ is conjugate to the block diagonal subgroup $\GL_n(\Q)\subset\SL_{n+1}(\Q)$. Fix $\ep\in\SL_{n+1}(\Q)$ so that 
\[\ep\> G_2(\Q)\>\ep^{-1}=\GL_n(\Q).\] 
Observe that there is a subgroup $\Lam<G_2(\Z)$ whose index in $G_2(\Z)$ is at most 2 and such that the Zariski closure of $\ep\Lam\ep^{-1}$ is
\[\SL_n=\left\{\left(\begin{array}{cc}A&0\\0&1\end{array}\right)\right\}\sbs\SL_{n+1}(\Q).\]
Specifically, if $\Q^n=P\oplus L$ is the decomposition preserved by $G_2(\Q)$, then $G_2(\Z)$ preserves $L\cap\Z^n\cong\Z$, and we can take $\Lam$ to be the kernel of the homomorphism $G_2(\Z)\ra\Aut(L\cap\Z^n)$. The Zariski closure of $\ep\Lam\ep^{-1}$ is certainly contained in $\SL_n$, and it is equal because $\ep\Lam\ep^{-1}$ contains a finite-index subgroup of $\SL_n(\Z)$.

Let $G_2'$ be the Zariski closure of $\Lambda$. To estimate the size of $W_2(r,s)$, it suffices to compute the image of $G_2'(\Z)\ra G_2'(\Z/t\Z)$ for $t=p^j$. In fact, we will show the following. 

\paragraph{Claim.} For all but finitely many $p$, the map $G_2'(\Z)\ra G_2'(\Z/t\Z)$ is surjective and the group $G_2'(\Z/t\Z)$ is isomorphic to $\SL_n(\Z/t\Z)$. 

From the claim we deduce (similar to the computation for $W(r,s)$) that 
\[|W_2(r,s)|\sim \left(p^{n^2-1}\right)^\ell.\]

\begin{proof}[Proof of Claim]First assume that $p$ does not divide the denominator of any entry of $\ep$ (this excludes finitely many $p$). Then $\ep$ descends to an invertible matrix in $\SL_{n+1}(\Z/t\Z)$. Observe that there is a finite index subgroup $\Lambda'<G_2'(\Z)$ so that $\ep\Lam'\ep^{-1}$ is contained in $\SL_n(\Z)$. Then the images of the maps 
\[\Lambda' \hra G_2'(\Z)\ra G_2'(\Z/t\Z)\>\>\>\text{ and }\>\>\>\ep\Lambda'\ep^{-1} \hra \SL_n(\Z)\ra \SL_n(\Z/t\Z)\] are conjugate in $\SL_{n+1}(\Z/t\Z)$. This proves the claim, since the map on the right is surjective for almost every $p$ by strong approximation \cite[Window 9, Cor.\ 3]{lubotzky-segal}.\footnote{More precisely, strong approximation implies that the inclusion $\ep\Lam\ep^{-1}\hra\SL_n(\Z)\hra\SL_n(\Z_p)$ has dense image. Therefore $\ep\Lam\ep^{-1}$ surjects onto the image of $\SL_n(\Z_p)\ra\SL_n(\Z/t\Z)$, and this latter map is surjective because $\SL_n(\Z)\ra\SL_n(\Z/t\Z)$ is surjective, as observed above.}
\end{proof}

\paragraph{Size of $W_1(r,s)$.} We show that $|W_1(r,s)|\sim \left(p^{n}\right)^\ell$. 

By the definition of $G_1$, the $R$-points $G_1(R)$ is the centralizer of $A_1$ in $\SL_{n+1}(R)$, for any ring $R$. Let $\chi\in\Z[x]$ be the characteristic polynomial of $A_1$. 


%
%
%

By Lemma \ref{lem:diagonalizable}, there are infinitely many primes $p$ so that $\chi$ splits completely in $(\Z/p^j\Z)[x]$ and $A_1$ is diagonalizable in $\GL_{n+1}(\Z/p^j\Z)$ for each $j\ge1$. Fix such a prime $p$, and set $R=\Z/p^j\Z$.
Since $A_1$ is diagonalizable with distinct eigenvalues, the centralizer of $A_1$ in $\GL_{n+1}(R)$ is conjugate to the group of diagonal matrices; hence this centralizer has size 
\[|(\Z/p^j\Z)^\ti|^{n+1}\sim (p^j)^{n+1}.\]

Next we show that $G_1(\Z)$ surjects to the centralizer of $A_1$ in $\SL_{n+1}(\Z/p^j\Z)$. The centralizer of $A_1$ in $M_{n+1}(\Z)$ contains a subring isomorphic to $\Z[t]/(\chi)$ (see \S\ref{sec:flats}). The same statement holds with $\Z$ replaced by $R$. Since $\chi$ splits completely in $R$ by our choice of $p$, there is a ring isomorphism $R[x]/(\chi)\cong R^{n+1}$. Up to conjugation the map $R[x]/(\chi)\ra M_{n+1}(R)$ has image equal to the diagonal matrices. From this we deduce that the invertible elements of $R[x]/(\chi)$ surject onto the diagonal subgroup of $\GL_{n+1}(R)$ and that $G_1(\Z)$ surjects onto the centralizer of $A_1$ in $\SL_{n+1}(\Z/p^j\Z)$. From this we conclude that $|W_1(r,s)|\sim \left(p^{n}\right)^\ell$.

\paragraph{An expression in terms of volume.} So far, we have shown that 
\[\frac{|W(r,s)|}{|W_1(r,s)|\cdot|W_2(r,s)|}\sim \left(p^{n^2+2n-(n^2-1)-n}\right)^\ell = \left(p^{n+1} \right)^\ell\]
for infinitely many primes $p$.

Since the volume satisfies $\vol(Y(s))=|W(r,s)|\cdot\vol(Y(r))$, we can write $\vol(Y(s)) \sim \left( p^{n^2+2n} \right)^\ell$. Therefore, 
\[\frac{|W(r,s)|}{|W_1(r,s)|\cdot|W_2(r,s)|}
\sim \left(p^{n+1} \right)^\ell
\sim \vol(Y(s))^{\fr{n+1}{n^2+2n}}.\]
Combined with Lemma \ref{lem:mrtbound} and Corollary \ref{cor:mrt-SL}, this proves Theorem \ref{thm:SL}. \qed

\subsection{$R_{F/\Q}(\SL_2)$ case (Theorem \ref{thm:ROS})}

In this section we prove Theorem \ref{thm:ROS} following the same outline as the preceding section. 

Fix $F\subset\R$ a totally real number field. Write $\ca H=\SL_2$, viewed as an algebraic group over $F$, and consider the restriction of scalars $G=R_{F/\Q}(\ca H)$. 

Given a prime ideal $\mf p\sbs\ca O_F$, let $\mf r=\mf p^k$ for some fixed $k\gg0$ large enough that 
Corollary \ref{cor:mrt-ROS} applies. For any $\ell>k$, let $\mf s=\mf p^\ell$. Next we estimate the size of the groups $W(\mf r,\mf s)=\Ga(\mf r)/\Ga(\mf s)$ and $W_i(\mf r,\mf s)=\Ga_i(\mf r)/\Ga_i(\mf s)$ for $i=1,2$ as functions of $\ell$ for infinitely many primes $\mf p$.

\paragraph{Size of $W(\mf r,\mf s)$.} 
First we show that $|W(\mf r,\mf s)|\sim N(\mf s)^3$,
where $N(\mf t)=|\ca O_F/\mf t|$ is the ideal norm. 
Similar to (\ref{eqn:congruence-SES}), we have a short exact sequence 
\[1\ra W(\mf r,\mf s)\ra\SL_2(\ca O_F)/\Ga(\mf s)\ra\SL_2(\ca O_F)/\Ga(\mf r)\ra1.\]

By \cite[Thm.\ 4.3.9]{hahn-omeara}, the reduction map $\SL_2(\ca O_F)\ra\SL_2(\ca O_F/\mf t)$ is surjective, so there is a short exact sequence
\begin{equation}\label{eqn:SES-ROS}1\ra\Ga(\mf t)\ra\SL_2(\ca O_F)\ra\SL_2(\ca O_F/\mf t)\ra1.\end{equation}



Using (\ref{eqn:SES-ROS}) and Proposition \ref{prop:SL-size}, we conclude that 
\[W(\mf r,\mf s)=\fr{|\SL_2(\ca O_F/\mf s)|}{|\SL_2(\ca O_F/\mf r)|}
\sim N(\mf s)^3.
\]

\paragraph{Size of $W_i(\mf r,\mf s)$.} 

Recall that $\ca H_i$ is the centralizer of $A_i\in\SL_2(F)$, and we have assumed that $A_1,A_2$ are conjugate in $\SL_2(F)$. As such, the computation of $|W_i(\mf r,\mf s)|$ is essentially the same for $i=1,2$. To simplify exposition, we present the proof for $i=1$ and show that $|W_1(\mf r,\mf s)| \sim N(\mf s)$.

We start by assuming $\mf p$ is not one of the finitely many factors of denominators of entries of $A_1$. This allows us to consider $A_1$ as an element of $\SL_2(\ca O_F/\mf p^k)$ for each $k\ge1$.



By Lemma \ref{lem:diagonalizable}, there are infinitely many prime ideals $\mf p\sbs\ca O_F$ such that the characteristic polynomial of $A_1$ splits completely in $(\ca O_F/\mf p^j)[x]$ and $A_1$ is diagonalizable in $\GL_2(\ca O_F/\mf p^j)$ for each $j\ge1$. Fixing such a prime, the centralizer of $A_1$ in $\GL_{2}(\ca O_F/\mf p^j)$ is conjugate to the group of diagonal matrices, so the centralizer of $A_1$ in $\SL_2(\ca O_F/\mf p^j)$ is isomorphic to $(\ca O_F/\mf p^j)^\ti$. Denoting $\mf t=\mf p^j$, this group has order $N(\mf t)\cdot\fr{N(\mf p)-1}{N(\mf p)}$, cf.\ Equation (\ref{eqn:unit-size}). 

Next we show that $\ca H_1(\ca O_F)$ surjects to the centralizer of $A_1$ in $\SL_2(\ca O_F/\mf p^j)$. The centralizer of $A_1$ in $M_{2}(\ca O_F)$ contains a subring isomorphic to $\ca O_F[x]/(\chi)$, where $\chi$ is the characteristic polynomial. The same statement holds with $\ca O_F$ replaced by $R:=\ca O_F/\mf p^j$. Since $\chi$ splits completely in $R$, there is a ring isomorphism $R[x]/(\chi)\cong R^2$. Up to conjugation the map $R[t]/(\chi)\ra M_2(R)$ has image equal to the diagonal matrices. From this we deduce that the invertible elements of $R[t]/(\chi)$ surject onto the diagonal subgroup of $\GL_2(R)$ and that $\ca H_1(\ca O_F)$ surjects onto the centralizer of $A_1$ in $\SL_2(R)$. 


From this we conclude that $|W_1(\mf r,\mf s)|\sim N(\mf s)$. By the same argument, we have $|W_2(\mf r,\mf s)|\sim N(\mf s)$.

\paragraph{An expression in terms of volume.} So far, we have shown that  
\[\frac{|W(\mf r,\mf s)|}{|W_1(\mf r,\mf s)|\cdot|W_2(\mf r,\mf s)|}\sim\fr{N(\mf s)^3}{N(\mf s)\cdot N(\mf s)} = N(\mf s)\]
for infinitely many primes $p$.

Since the volume satisfies $\vol(Y(\mf s))=|W(\mf r,\mf s)|\cdot\vol(Y(\mf r))$, we can write $\vol(Y(s))\sim N(\mf s)^{3}$. Therefore, $\frac{|W(\mf r,\mf s)|}{|W_1(\mf r,\mf s)|\cdot|W_2(\mf r,\mf s)|}\sim\vol(Y(\mf s))^{1/3}$. 
Combined with Lemma \ref{lem:mrtbound} and Corollary \ref{cor:mrt-ROS}, this proves Theorem \ref{thm:ROS}. \qed

\subsection{Orthogonal case (Theorem \ref{thm:SO})}\label{sec:SO-proof}

 In this section we prove Theorem \ref{thm:SO}. 
Fix $n\ge2$ and let $G$ denote the algebraic group
\[\SO(Q_n)=\{g\in\SL_{2n}: g^tQ_ng=Q_n\},\]
where $Q_n$ is the matrix defined in (\ref{eqn:hyperbolic-form}). 

Take $G_1,G_2<G$ as in \S\ref{sec:flat-homology-SO}. For any prime $p\equiv 3\pmod{4}$, let $r=p^k$ for some fixed $k\gg 0$ large enough that 
Corollary \ref{cor:mrt-SO} applies. For any $\ell>k$, let $s=p^\ell$. Below we estimate the size of the groups $W(r,s)=\Ga(r)/\Ga(s)$ and $W_i(r,s)=\Ga_i(r)/\Ga_i(s)$ for $i=1,2$ as functions of $\ell$ for infinitely many primes $p$.

\paragraph{Size of $W(r,s)$.} 
First show that $|W(r,s)|\sim \left(p^{n^2}\right)^\ell$. Recall the following short exact sequence (\ref{eqn:congruence-SES}) from \S\ref{sec:geocycles}: 
\[1\ra W(r,s)\ra G(\Z)/\Ga(s)\ra G(\Z)/\Ga(r)\ra1.\]
For $t=p^j$, the group $G(\Z)/\Ga(t)$ is isomorphic to the image of 
\[\SO(Q_n;\Z)\ra\SO(Q_n;\Z/t\Z).\] 
By strong approximation and Hensel's lemma, the image of this map contains the group $\Om(Q_n;\Z/t\Z)$ defined in (\ref{eqn:omega}) for almost every prime $p$. This implies that $|W(r,s)|\sim \left( p^{n^2} + 2 p^{n(n-1)/2}\right)^\ell$ by Corollary \ref{cor:omega}. 

Our application of these results is routine, so we only give a brief explanation. Similar to \S\ref{sec:SL}, strong approximation is used to conclude that the image of the inclusion $G(\Z)\hra G(\Z_p)$ is dense. One difference is that the algebraic group $\SO(Q_n)$ is not simply connected, so in order to apply strong approximation, one needs to first pass to the 2-fold spin cover; see \cite[Window 9, \S1 and \S4]{lubotzky-segal} and \cite{pink}. Hensel's lemma in algebraic geometry implies that the reduction map $G(\Z_p)\ra G(\Z/p^j\Z)$ is surjective for each $j\ge1$; see \cite[\S18.5]{grothendieck} and \cite[Thm.\ 4.17]{stulemeijer}. In order to apply Hensel, one needs $G$ to be a smooth $\Z_p$-scheme. The smoothness of $G=\SO(Q_n)$ as a group scheme is discussed in \cite[C.1.5]{conrad}.

\paragraph{Size of $W_2(r,s)$.} For all but finitely many primes $p$ (depending on the choice of the negative vector), the group $G_2(\Z/t\Z)$ is isomorphic to $\SO(\wtil Q_{n-1},\Z/t\Z)$, where $\wtil Q_n$ is defined in (\ref{eqn:form2}). Again by strong approximation, the image of $G_2(\Z)\ra G_2(\Z/t\Z)$ is isomorphic to $\Om(\wtil Q_{n-1};\Z/t\Z)$ for almost every prime $p$. This implies that $|W_2(r,s)|\sim \left( p^{(n-1)^2} + 2 p^{(n-1)(n-2)/2 + p^{n-1}}\right)^\ell$ by Corollary \ref{cor:omega2}.

\paragraph{Size of $W_1(r,s)$.} We show that $|W_1(r,s)|\sim \left(p^n\right)^{\ell}$. 

By the definition of $G_1$, the group $G_1(R)$ is the centralizer of $A_1$ in $\SO(Q_n;R)$, for any ring $R$. Let $\chi\in\Z[x]$ be the characteristic polynomial of $A_1$. 

By Proposition \ref{prop:splitting} and Lemma \ref{lem:diagonalizable-SO}, there are infinitely many primes $p\equiv3\pmod{4}$ so that $\chi$ splits completely in $(\Z/p^j\Z)[x]$ and $A_1$ is diagonalizable in $\SO(Q_n;\Z/p^j\Z)$ for each $j\ge1$. Fix such a prime $p$, and set $R=\Z/p^j\Z$. Since $A_1$ is diagonalizable with distinct eigenvalues, the centralizer of $A_1$ in $\SO(Q_n;R)$ is conjugate to the group of diagonal matrices in $\SO(Q_n;R)$; hence this centralizer has size 
\[|(\Z/p^j\Z)^\ti|^{n}\sim (p^j)^{n}.\]

It remains to show that $G_1(\Z)$ surjects to the centralizer of $A_1$ in $\SO(Q_n;\Z/p^j\Z)$. The centralizer of $A_1$ in $M_{2n}(\Z)$ contains a subring isomorphic to $\Z[t]/(\chi)$ (see \S\ref{sec:flats}). The same statement holds with $\Z$ replaced by $R$. Since $\chi$ splits completely in $R$ by our choice of $p$, there is a ring isomorphism $R[x]/(\chi)\cong R^{n+1}$. Up to conjugation the map $R[x]/(\chi)\ra M_{2n}(R)$ has image equal to the diagonal matrices. From this we deduce that the invertible elements of $R[x]/(\chi)$ surject onto the diagonal subgroup of $\GL_{2n}(R)$ and that $G_1(\Z)$ surjects onto the centralizer of $A_1$ in $\SO(Q_n;\Z/p^j\Z)$. From this we conclude that $|W_1(r,s)|\sim \left(p^n\right)^{\ell}$ (for an isometry, the diagonal entries occur in $(\la,\frac{1}{\la})$ pairs).

\paragraph{An expression in terms of volume.} Combining the above estimates, we have 
\[
\frac{|W(r,s)|}{|W_1(r,s)|\cdot|W_2(r,s)|}
\sim \left(\frac{p^{n^2}+p^{n(n-1)/2}}{p^{n^2-n+1} + 2 p^{(n^2-n+2)/2} + p^{2n-1}}\right)^\ell
\]
for infinitely many primes $p \equiv 3 \pmod{4}$.

Since the volume satisfies $\vol(Y(s))=|W(r,s)|\cdot\vol(Y(r))$, we can write $\vol(Y(s))\sim \left(p^{n^2}+p^{n(n-1)/2} \right)^{\ell}$
Therefore, 
\[\frac{|W(r,s)|}{|W_1(r,s)|\cdot|W_2(r,s)|}
\sim \vol(Y(s))^\kappa,\]
where 
\[
\kappa = 1 - \frac{\log(p^{n^2-n+1} + 2 p^{(n^2-n+2)/2} + p^{2n-1})}{\log(p^{n^2}+p^{n(n-1)/2})}.
\]
For given $p$ and $n$, the numerator of the above fraction is bounded above by $\log(3 p^{n^2-n+1})$ while the denominator is bounded below by $\log(p^{n^2})$, so there is an inequality 
\[
\kappa \geq \frac{n-1 - \log_p(3)}{n^2}.
\] 
Combined with Lemma \ref{lem:mrtbound} and Corollary \ref{cor:mrt-SO}, this proves Theorem \ref{thm:SO}. \qed

%
%
%

\bibliographystyle{amsalpha}
\bibliography{refs}

Daniel Studenmund\\
Department of Mathematical Sciences, Binghamton University\\
\texttt{daniel@math.binghamton.edu}

Bena Tshishiku\\
Department of Mathematics, Brown University\\ 
\texttt{bena\_tshishiku@brown.edu}

\end{document}